\documentclass[11pt,a4paper, reqno]{amsart}

\date{\today}

\usepackage{amssymb, amsmath, amsfonts, latexsym, enumerate}
\usepackage[english]{babel}
\usepackage{comment}
\usepackage[utf8]{inputenc}
\usepackage{hyperref}

\usepackage{url}
\usepackage{amsmath}
\usepackage{amsfonts}
\usepackage{amssymb}
\usepackage{amsthm}
\usepackage{dsfont}
\usepackage{comment}
\usepackage{mathtools}
\usepackage{graphicx}
\newtheorem{thm}{Theorem}[section]
\newtheorem{lemma}[thm]{Lemma}
\newtheorem{prop}[thm]{Proposition}
\newtheorem{cor}[thm]{Corollary}

\newtheorem{definition}[thm]{Definition}

\usepackage{color}
\theoremstyle{remark}
\newtheorem{remark}[thm]{Remark}

\newtheorem{example}[thm]{Example}
\newtheorem{counterex}[thm]{Counterexample}

\numberwithin{equation}{section}
\newcommand{\K}{\mathbf{k}}
\newcommand{\LF}{\mathfrak{F}_{\mathbf{A},\nv} }
\newcommand{\RD}{\mathbb{R}^d}
\newcommand{\R}{\mathbb{R}}
\newcommand{\Z}{\mathbb{Z}}
\newcommand{\N}{\mathbb{N}}

\newcommand{\Rec}{\mathfrak{R}}
\newcommand{\drecl}{\mathcal{R}_{\mathbf{A}}(\pl)}
\newcommand{\drec}{\mathcal{R}_{\mathbf{A}}(\mu)}
\newcommand{\smch}{\mathcal{J}}
\newcommand{\tinv}{\Xi^{\nu} }
\newcommand{\sinv}{\Phi_{\nu}^{\cv}}
\newcommand{\PM}{\mathcal{P}}
\newcommand{\PR}{\mathbb{P}}
\newcommand{\PN}{\mathbb{P}_{\nu}}
\newcommand{\PNX}{\mathbb{P}^{x}_{\nu}}
\newcommand{\PNB}{\mathbb{P}_{\nu}^{xy}}

\newcommand{\BM}{\mathcal{B}}
\newcommand{\pl}{\mathfrak{p}_{\lambda}}
\newcommand{\om}{\omega}
\newcommand{\QR}{\mathbb{Q}}

\newcommand{\A}{\mathcal{A}}
\newcommand{\C}{\ker_\Z(\mathbf{A})}
\newcommand{\PO}{\mathcal{P}(\Omega)}
\newcommand{\GCC}{G_{\mathbf{c}}}
\newcommand{\GC}{G_{\vv}}

\newcommand{\GT}{\exp \Bigl( \sum_{j=1}^{A}\int_{0}^{1} \log \Xi^{\nu}(j,t,u(j,t)) \ \dot{u}(j,t) dN_t^j \Bigr) }
\newcommand{\nv}{\mathbf{n}}
\newcommand{\cv}{\mathbf{c}}
\newcommand{\vv}{\mathbf{v}}

\newcommand{\CS}{\Gamma}

\newcommand{\zv}{\mathbf{z}}
\newcommand{\mv}{\mathbf{m}}
\newcommand{\NO}{\mathbf{N}_1}
\newcommand{\cy}{\gamma}
\newcommand{\trc}{\Upsilon_{\varepsilon}}
\newcommand{\CST}{\ker^{*}_{\Z}(\mathbf{A})}
\newcommand\ps{\textrm{-a.s.}}

\newcommand\OO{\Omega}

\newcommand{\s}{\mathcal{S}}

\newcommand{\U}{\mathcal{U}}
\newcommand{\shv}{\theta_{\vv}}
\newcommand{\shc}{\theta_{\cv}}
\newcommand{\nmin}{n^-}
\newcommand{\npiu}{n^+}

\begin{document}

\title{Reciprocal class of jump processes}

\author{Giovanni Conforti}
\address{Institut f\"ur Mathematik der Universit\"at Potsdam. Am Neuen Palais 10. 14469 Potsdam, Germany}
\email{giovanniconfort@gmail.com}

\author{Paolo Dai Pra}
\address{Universit\'{a} degli Studi di Padova, Dipartimento di Matematica Pura e
Applicata. Via Trieste 63, 35121 Padova, Italy}
\email{daipra@math.unipd.it}

\author{Sylvie R\oe lly}
\address{Institut f\"ur Mathematik der Universit\"at Potsdam. Am Neuen Palais 10. 14469 Potsdam, Germany}
\email{roelly@math.uni-potsdam.de}

%\keywords{
%Reciprocal processes, stochastic bridges, jump processes, compound Poisson processes.}
%
%\MSC 60G55, 60H07, 60J75
%
%
%\date{\today}
%

\begin{abstract} 

Processes having the same bridges as a given reference Markov process constitute its {\it reciprocal class}. In this paper we study the reciprocal class of 
compound Poisson processes whose jumps belong to a finite set $\mathcal{A} \subset \mathbb{R}^d$. We propose a characterization of the reciprocal class as  the unique set of
probability measures on which a family of time and space transformations induces the same density, expressed in terms of the \textit{reciprocal invariants}. 
The geometry of $\mathcal{A}$ plays a crucial role in the design of the  transformations, and we use tools from discrete 
geometry to obtain an optimal characterization.
We  deduce explicit conditions for two Markov jump processes to belong to the same class. Finally, we provide a natural interpretation of the invariants 
as short-time asymptotics for the probability that the reference process makes a cycle around its current state.
\end{abstract}

\maketitle

\maketitle
\tableofcontents

\section*{Introduction}

For a given $\RD$-valued  stochastic process $X = (X_t)_{t \in [0,1]}$ and $I \subseteq [0,1]$ we let   $\mathcal{F}_I$ be the $\sigma$-field generated by the 
random variables $(X_s: s \in I)$. We say that $X$ is a {\em reciprocal process} if for every $0 \leq s < t \leq 1$ the $\sigma$-fields $\mathcal{F}_{[0,s] \cup
[t,1]}$ and $\mathcal{F}_{[s,t]}$ are independent {\em conditionally} to $\mathcal{F}_{\{s,t\}}$. Comparing this notion to that of Markov process
($\mathcal{F}_{[0,t)}$ and $\mathcal{F}_{(t,1]}$ independent conditionally to $\mathcal{F}_{\{t\}}$) it is not hard to show that every Markov process is
reciprocal, but non-Markov reciprocal processes exist (see e.g. \cite{Roell13}). 

The notion of reciprocal process is very natural in many respects. On one hand it emerges when one solves dynamic problem such as 
stochastic control problems or stochastic differential equations with boundary constraints, i.e. constraints on the joint distribution at the boundary times
$t=0$ and $t=1$; this point of view has actually inspired the whole theory of reciprocal processes, that originated from ideas in \cite{Schr} and \cite{Bern32}
and led to several developments (see e.g. \cite{Zam86, Wak89, DP91}).
On the other hand it is a special case of the more general notion of {\em Markov random field} (\cite{Cre93}), 
which provides a Markov property for families of random variables $(X_r)$ indexed by $r \in \R^d$.

The systematic study of reciprocal processes has initiated with the Gaussian case: 
covariance functions giving rise to reciprocal Gaussian processes have been thoroughly studied and characterized (\cite{Jam70, Jam74, Jam75, Cha72,
CMT92,Lev97}). A more ambitious aim has been that of describing reciprocal processes in terms of {\em infinitesimal characteristics}, playing the role that
%Stochastic Differential Equations and 
infinitesimal generators play for Markov processes; this has led to the introduction of a second order stochastic calculus
(\cite{Kre88, Th93, Kre97}).

In this paper we consider a related problem.
%, though of more limited scope. 
Suppose we are given a {\em reference} Markov process, whose law on its path space will be denoted by $\PR$. For simplicity, we
assume $X$ to be the canonical process on its path space. A probability $\QR$ is said to belong to the {\em reciprocal class} of $\PR$ if for every $0 \leq s <
t \leq 1$ and $A \in \mathcal{F}_{[s,t]}$ we have
\begin{equation} \label{ref-class}
\QR(A|\mathcal{F}_{[0,s] \cup [t,1]}) = \PR(A|\mathcal{F}_{[0,s] \cup [t,1]}) = \PR(A|X_s,X_t),
\end{equation}
where the last equality is an immediate consequence of the fact that $X$, being Markov, is a reciprocal process under $\PR$. 
In particular, $X$ is a reciprocal process also under $\QR$. The elements of the reciprocal class of $\PR$ are easy to characterize from a measure-theoretic
point of view. Denote by $\PR^{xy}$  the {\em bridge} of $\PR$ from $x$ to $y$, i.e. the probability obtained by conditioning $\PR$ on $\{(X_0,X_1)=(x,y)\}$; a
probability
$\QR$ is in the reciprocal class of $\PR$ if and only if it is a mixture of the bridges of $\PR$, i.e.
\[
\QR = \int \PR^{xy} \mu(dx,dy)
\]
for some probability $\mu$ on $\RD \times \RD$. 

Along the lines of what we have discussed above, it is desirable to characterize the probability measures belonging to the reciprocal class of $\PR$ by
infinitesimal characteristics. 
One first question in this direction is the following. 
Assume $\PR$ is the law of a Markov process with infinitesimal generator $L$. 
Given another Markov generator $L'$, under what conditions the laws of the Markov processes generated by $L'$ belong to the reciprocal class of $\PR$? 
This question is well understood for $\RD$-valued diffusion processes with smooth coefficients, and it has motivated the search for the so-called {\em
reciprocal invariants}: the collection of reciprocal invariants forms a functional $F(L)$ of the coefficients of the generator, such that the following
statement holds: the laws of the processes generated by $L$ and $L'$ belong to the same reciprocal class if and only if $F(L) = F(L')$. 
Explicit expressions for
the reciprocal invariants of diffusion processes can be found in \cite{Kre88, Cl91, LeKre93}. 
For pure jump processes with values in a discrete state space, the understanding of reciprocal invariants is very limited, except for some
 special cases like counting processes (\cite{LMR},\cite{Roell13}) or for pure jump processes with independent
increments under very restrictive assumption on the geometry of jumps (called incommensurability), see Chapter 8 in \cite{Murr12}. 

In this paper we consider possibly time-inhomogeneous compound Poisson processes
with jumps in a finite set $\A$,  considerably weakening the assumptions in \cite{Murr12}. Our analysis reveals how
reciprocal invariants are related to the geometric and graph-theoretic structure induced by the jumps.
Close ideas apply to other context where a similar
structure emerges, in particular to random walks of graphs, which will be treated in a forthcoming paper. 
To make clearer the improvement with respect to \cite{Murr12}, we
note that the assumption there imply that the corresponding graph structure in {\em acyclic}; 
In our framework, we will see that {\em cycles} are exactly the main parameter in the
collection of reciprocal invariants, see Proposition \ref{anytime}.

The basic tool for identifying the reciprocal invariants is provided by functional equations, called duality formulae or integration by parts formulae, which represent a subject of independent interest. 
%Duality formulae, in the form of e in which a stochastic integral operator appears as dual of a Malliavin derivative, 
They have provided in particular useful characterizations of 
%the Wiener measure and Markov diffusions (\cite{Bis81, RZ91}), and of 
Poisson processes (\cite{Sli62, Mec67}). 
The idea of restricting the class of test functions in the duality formula in order to characterize the whole reciprocal class has
appeared for the first time in \cite{RT02,RT05} in the framework of diffusions. In this paper we make explicit a functional equation containing a difference operator, which only depend on reciprocal invariants and characterize the reciprocal class of compound Poisson processes. 
%One of the differences with previous works is that 
% Our duality formula is in non-differential form, i.e. it contains a
% difference operator takes the place of a Malliavin derivative. 

The paper is organized as follows. In Section \ref{sec:framework} we set up the necessary notations and provide the relevant definitions. 
In Section \ref{sec: trans } we define suitable transformations on the path space, and compute the density of the image of the law of the process under
these transformations. This will allow in Section \ref{sec:char} to derive the duality formulae and to identify the reciprocal invariant. At the end of Section
\ref{sec:char} we also give an interpretation of the reciprocal invariants in the time-homogeneous case, in terms on the asymptotic probability the process
follows a given cycle. This could be extended to other contexts, e.g. to Markov diffusions, providing an alternative to the physical interpretation given in
\cite{Cl91}. These extensions will be the subject of a forthcoming work.

%%%%%%%%%%%%%%%%%%%%%%%%%%%%%%%%%%%%%%%%%%%%%%%%%%%%%%%%%%%%%%%%%%%%%%%%
\section{Framework. Some definitions and notations}\label{sec:framework}
%%%%%%%%%%%%%%%%%%%%%%%%%%%%%%%%%%%%%%%%%%%%%%%%%%%%%%%%%%%%%%%%%%%%%%%%
%%%%%%%%%%%%%%%%%%%%%%%%%%%%%%%%%%%%%%%%%%%%%%%%%%%%%%%%%%%%%%%%%%%%%%%%

We consider $\RD$-valued random processes with finitely many types of jumps, chosen in a given set  
\begin{equation}\label{jumpmatrix}
\A= \left\{ a^{1},...,a^{A} \right\} \subseteq \RD
\end{equation}
of cardinality $A$.
We associate to  $\A$ the  matrix 
$\mathbf{A}=( a^{j}_{i})_{1 \leq i \leq d, 1 \leq j\leq A} \in \mathbb{R}^{d \times A}$.
Their paths are elements of $\mathbb{D}([0,1],\RD)$, the space of right continuous with left limits functions on $[0,1]$ (usually called in french {\em
c\`adl\`ag} paths), equipped
with its canonical filtration $\left( \mathcal{F}_{t} \right)_{t \in [0,1]}$. $\mathcal{F}_{[s,t]}$ is defined by $\sigma( ( X_r)_{r \in [s,t]} )$
%We denote by $X$ the canonical process and by $X_t$ the projection at time $t$.
%%%%%%%%%%%%%%%%%%%%%%%%%%%%%%%%%%%%%%%%%%%%%%%%%%%%%%%%%%%%%%%%%%%%%%%%%%%%%%%%%%%%%%%%%%%%%%%%%%%%%%%%%%%%%%%%%%%%%%%%%%%%%%%%%%%%%%%%%%%%%%%%%%%%%%%%%%%%%%%%%%%%%%%

In this setting, paths can be described by the counting processes corresponding to each type of jumps. 
It is therefore natural to introduce the following random variables:
\begin{definition}
Let define  $\mathbf{N}=(\mathbf{N}_t)_{0 \leq t \leq 1}$, where $\mathbf{N}_t:=(N_t^{1},...,N_{t}^{A})$ and, for any $j \in \{1,...,A\},  N^j_t$ counts
how many times the jump $a^j$ has occurred up to time $t$:
  \begin{eqnarray*}
 N^{j}_{t}(\om)=  \sum_{s \leq t} \mathbf{1}_{\left\{ \om_{s} - \om_{s^{-}} = a^{j} \right\} }.
 \end{eqnarray*}
The total amount of jumps  up to time $t$, $ |N|_{t}$,  is given by the sum of the coordinates of $\mathbf{N}_t$, that is 
$ |N|_{t}:= \sum_{j=1}^{A} N^j_t$. \\
The $i$-th jump time of type  $a^j$ is:
\begin{equation*}
 T_{i}^{j}:= \inf \left\{ t \in [0,1]: N^{j}_{t} =i \right\} \wedge 1.
 \end{equation*}
Finally, the $i^{th}$ jump time of the process is:
\begin{equation*}
 T_i:= \inf \{ t \in [0,1]: \ |N|_t = i \} \wedge 1.
\end{equation*}

 \end{definition}

%In the next sections, it will be convenient to view  $\mathbf{N}_t$ as a  $\Z^A$ valued random vector.\\
 Then, we can express the canonical process as $X_t = X_0 + \sum_j a^j N^j_t$, which leads to introduce the following 
  set $\OO$ of paths indeed carried by the  processes we consider here:
\begin{equation*}
\begin{split}
\OO & = \big\{ \om : |N|_1(\om) <+ \infty \textrm{ and } X_t(\om)=  X_0(\om) + \mathbf{A} \mathbf{N}_t(\om), \ 0 \leq t \leq 1 \big\} \\ &  \subseteq
\mathbb{D}([0,1],\RD) .
\end{split}
\end{equation*}

%Moreover, $\mathcal{F}_{\nv,x}$ is the sigma algebra induced by $\mathcal{F}$ on $\OO_{\nv,x}$.

We also define the set $ \s $ of possible initial and final values $(X_0,X_1) $ for paths in $\OO$:
\begin{equation*}\label{margsupport}
\s:= \left\{ (x,y) \in \RD \times \RD, \exists \ \nv \in \N^{A} \  \textrm{such that} \ y= x + \mathbf{A}\nv \right\} .
\end{equation*}
 %%%%%%%%%%%%%%%%%%%%%%%%%%%%%%%%%%%%%%%%%%%%%%%%%%%%%%%%%%%%%%%%%%%%%%%%%%%%%%%%%%%%%%%%%%%%%%%%%%%%%%%%%%%%%%%%%%%%%%%%%%%%%%%%%%%%%%%%%%%%%%%%%%%%%%%%%%%%%%%%%%%%%%%%%%%%%%%%%%%%%%%%%%%%%%%%%%%%%%%%%%%%%%%

 \vspace{0.2cm}
For any measurable space $\mathcal{X}$, we will denote by  $\mathcal{M}(\mathcal{X})$ 
the set of all non negative measures on $\mathcal{X}$ and 
by $\PM (\mathcal{X})$ the subset of  probability measures on $\mathcal{X}$.
 $\BM(\mathcal{X})$ is the set of bounded measurable functions on  $\mathcal{X}$.

A general element of $\PM(\OO)$ will be denoted by $\QR$.
Concerning its time projections, we use the notations
\begin{equation*}
\QR_{t}:=\QR \circ X_t^{-1} \textrm{ and } \quad \QR_{01} : = \QR\circ (X_0,X_1)^{-1} 
\end{equation*}
for the law at time $t$, resp. the joint law at time $0$ and $1$.
%%%%%%%%%%%%%%%%%%%%%%%%%%%%%%%%%%%%%%%%%%%%%%%%%%%%%%%%%%%%

As reference process 
%with paths in $\OO$
, we will consider in this paper a time-inhomogeneous compound Poisson process 
denoted by $\PNX$, 
where  $x \in \RD $ is a fixed initial position and $\nu$ is a regular jump measure belonging to the set  
%%%%%%%%%%%%%%%%%%%%%%%%%%%%%%%%%%%%%%%%%%%%%%%%%%%%%%%%%%%%%%%%%%%%
\begin{multline}\label{ eq: charact}
\smch : = \Big\{   \nu \in \mathcal{M}(\A \times [0,1] ), \ \nu (dx,dt) =  \sum_{j=1}^A  \delta_{a^j} (dx)\otimes \nu^{j}(t) dt,  \\
 \nu^{j}(\cdot) \in
C([0,1],\R_{+}), 1\leq j \leq A \Big\} .
\end{multline}
%%%%%%%%%%%%%%%%%%%%%%%%%%%%%%%%%%%%
 Heuristically, the process with law $\PNX$ leaves its current state at rate $\sum_{j=1}^{A} \nu^j$ and when it jumps, 
it chooses the jump $a^j$ with probability $\nu^j/\sum_{j=1}^{A} \nu^j$.
More precisely, under $\PNX$ the canonical process $X$ satisfies  $X_0 = x$ a.s. and
it has independent increments whose distribution is determined by its characteristic function
\begin{equation} \label{laptransf}
\PNX \Big( \exp \big(i \lambda \cdot (X_t - X_s)\big) \Big) = \exp\Big( \sum_{j=1}^A \big(e^{i \lambda \cdot a^j} -1 \big) \int_s^t  \nu^j(r) dr \Big), \
\lambda \in \RD,
\end{equation}
where $\lambda \cdot x$ is the scalar product in $\RD$. Note that here, as well as in the rest of the paper, 
for  $\QR \in \PM(\OO)$ and $F: \OO \rightarrow \mathbb{R}$, we write $\QR(F)$ for $\int F d\QR$.\\
 The properties 
of $\PNX$ are well known, in particular its semimartingale characteristics $(0,\nu,0 )$, see e.g. Chapters II and III  of \cite{JaShi03}.
%, Theorem 4.15. (see e.g.  point c) of Theorem 2.16 in Chapter III  are 

%here we always consider  the cutoff function $h=\mathds{1}_{ \left\{x >0 \right\} }$).
%%%%%%%%%%%%%%%%%%%%%%%%%%%%%%%%%%%%%%%%%%%%%%%%%%%%%%%%%%%%%%%%%%%%%%%%%%%%%%%%%%%%%%%%%%%%%%%%%%%%%%%%%%%%%%%%%%%%%%%%%%%%%%%%%%%%%%%%%%%%%%%%%%%%%%%%%%%%%%%%%%%%%%%%%%%%%%%%%%%%%%%%%%%%%%%%%%%%%%%%%%%%

%%%%%%%%%%%%%%%%%%%%%%%%%%%%%%%%%%%%%%%%%%%%%%%%%%%%%%%%%%%%%
\subsection{Reciprocal classes}
We first define a bridge of the compound Poisson process $\PNX$.
%%%%%%%%%%%%%%%%%%%%%%%%%%%%%%%%%%%%%%%%%%%%%%%%%%%%%%%%%%%
\begin{definition}\label{def:bridge} 
For $(x,y) \in \s$ and $\nu \in \smch $, $\PNB$, the bridge of the compound Poisson process from $x$ to $y$, is given by the probability measure on $\OO$:
\begin{equation*}\label{eq:bridge}
 \PNB :=\PNX(\ .\ |X_{1}=y).
\end{equation*}
\end{definition}
%%%%%%%%%%%%%%%%%%%%%%%%%%%%%%%%%%%%%%%%%%%%%%%%%%%%%%%%%%%%
%%%%%%%%%%%%%%%%%%%%%%%%%%%%%%%%%%%%%%%%%%%%%%%%%%%%%%%%%%%%

\begin{remark}
Note that $\PNB$ is well defined as soon as $(x,y) \in \s $, since in that case $\PNX(X_1=y)>0$.
\end{remark}

%%%%%%%%%%%%%%%%%%%%%%%%%%%%%%%%%%%%%%%%%%%%%%%%%%%%%%%%%%%
The reciprocal class associated to the jump measure $\nu \in \smch$ is now defined  as the set of all possible mixtures of  bridges in the family 
$(\PNB)_{(x,y) \in \s }$.
%%%%%%%%%%%%%%%%%%%%%%%%%%%%%%%%%%%%%%%%%%%%%%%%%%%%%%%%%%%%% 
\begin{definition}\label{def-clasrec} 
Let  $\nu \in \smch $. Its associated \emph{reciprocal class} is the following set of probability measures on $\OO$:
 \begin{equation*}\label{eq:clasrec}
 \Rec ( \nu ) := \big\{ \QR \in \PO :  \QR(\cdot)=\int_{ \s  } \PNB (\cdot) \, \QR_{01}(dxdy) \big\}.
\end{equation*}
\end{definition}

%Probability measures  $ \QR $ belonging to a reciprocal class show an interesting structural time symmetry, called {\it reciprocal  property} 
%(or two-side Markov property):
%take any sucessive times $ 0<s\le u\le 1$ and condition $\QR$ with respect to the past of $s$ and the  future of $u$. This is equivalent to condition $\QR$ 
% knowing only the marginals at both times $s$ and $u$.
%Bernstein first underlines this property in the particular framework of diffusion processes in his talk \cite{Bern32} at the International Congress in Z\"urich
%in 1932. For this historical reason, reciprocal processes are sometimes called Bernstein processes.
%\\ 
%Nevertheless the reciprocal property is weaker than the Markov one, and we will discuss in the rest of the paper
% cases of probability measures which are non Markov but are reciprocal. See \cite{LRZ} for more details about the relation between Markov and reciprocal properties.  

Let us describe the specific structure of any probability measure in  the reciprocal class $\Rec ( \nu )$.

%%%%%%%%%%%%%%%%%%%%%%%%%%%%%%%%%%%%%%%%%%%%%%%%%%%
\begin{prop}\label{prop:Q=hP}
Let $ \QR \in \PO $. Define then the compound Poisson process $\PN^{\QR} $ with the same dynamics as $\PNX$ but the same initial distribution  as $\QR$ by
$ \PN^{\QR} = \int_{\RD} \PNX(.) \QR_{0}(dx).$
Then the following assertions are equivalent:
\begin{itemize}
\item[$i$)]  $\QR \in \Rec(\nu)$
\item[$ii$)] $\QR$ is absolutely continuous with respect to $\PN^{\QR}$ and the density $d\QR/d\PN^{\QR}$ is 
$\sigma(X_0,X_1)$-measurable.
\end{itemize}
\end{prop}
%%%%%%%%%%%%%%%%%%%%%%%%%%%%%%%%%%%%%%%%%%%%%%%%%%%%%%%%%%%%%%%%%%%%%%%%%%%%%%%%%%%%%%%%%%%%%%%%%%%%%%%%%%%%%%%%%%%%%%%%%%%%%%%%%%%%%%%%%%%%%%%%%%%%%%%%%%%%%%%%%%%%%%%%%%%%%%%%%%%%%%%%%%%%%%%%%%%%%
\begin{proof}
$i$) $\Rightarrow ii$)\\
We first prove that  
 $\QR_{01}$ is absolutely continuous with respect to $(\PN^{\QR})_{01}$.\\
 Let us consider a Borel set $O \subseteq \R^d \times \R^d$ such that $\QR_{01}(O)>0$. 
 There exists $\nv \in \N^A$ such that $ \QR_{01}(O \cap \{ \mathbf{N}_1=\nv\} )>0$.
We can rewrite this event in a convenient way:
\begin{equation*}
\left\{ (X_0,X_1) \in O \right\} \cap  \{ \mathbf{N}_1=\nv\}   = \left\{ X_0 \in \tau_\nv^{-1}(O) \right\} \cap \  \{ \mathbf{N}_1 = \nv\} 
\end{equation*}
where  $\tau_\nv: \R^d \rightarrow \R^d \times \R^d$ is the map $ x \mapsto (x,x+\mathbf{A}\nv) $.
Since $\QR_{01}(O \cap \{ \mathbf{N}_1=\nv\} ) >0$, $\QR_{0}( \tau_\nv^{-1}(O)) = ({\PN})_{0}( \tau_\nv^{-1}(O)) >0$. \\
A simple application of the Markov property of $\PN^{\QR}$ implies 
that
\begin{equation*}
\begin{split}
(\PN^{\QR})_{01} \left( O \right) &  \geq \PN^{\QR} \left( (X_0,X_1) \in O \cap \{ \mathbf{N}_1=\nv \} \right) \\ & = (\PN^{\QR})_0( \tau_\nv^{-1}(O) ) \PN^{\QR}( \mathbf{N}_1=\nv)>0.
\end{split}
\end{equation*}
Therefore we can conclude that $\QR_{01}<<(\PN^{\QR})_{01}$ and we denote by $h$ the density function  $d\QR_{01} / d(\PN^{\QR})_{01}$.\\
Finally, let us choose any $F \in \BM(\Omega)$. By hypothesis:
\begin{eqnarray*}
\QR \big( F \big) &=& \QR( \QR[ F | X_0,X_1 ]  ) = \QR \big( \PN^{X_0,X_1} (F) \big) \\
&=& \PN^{\QR} \big( \PN^{X_0,X_1}(F) h(X_0,X_1) \big) = \PN^{\QR} \big( \PN^{X_0,X_1} (F h(X_0,X_1) ) \big) \\ & =&  \PN^{\QR} \big(F h(X_0,X_1) \big),
\end{eqnarray*}
which leads to the conclusion that
\[
\frac{d\QR}{d\PN^{\QR}} =  \frac{d\QR_{01}}{d(\PN^{\QR})_{01}} =  h(X_0,X_1).
\]
This proves $ii$).\\
$ii$) $\Rightarrow i$)
Suppose that there exists a non negative measurable function $h$ such that $d\QR / d\PN^{\QR}=h(X_0,X_1)$.
It is a general result in the framework of reciprocal processes that, in that case,  $\QR \in \Rec(\nu)$, see e.g. Theorem 2.15 in \cite{LRZ}. 
For sake of completeness, we recall shortly the arguments. Let us take three measurable test functions  $\phi, \psi, F$.
\begin{eqnarray*}
\QR \left(  \phi(X_0) \psi(X_1) F  \right) & = & \PN^{\QR} \left(  \phi(X_0) \psi(X_1) h(X_0,X_1) F \right) \\
 	&=&  \PN^{\QR} \left(  \phi(X_0) \psi(X_1) h(X_0,X_1) \PN^{\QR} (F|X_0,X_1 ) \right)\\
	&=&\QR \left(  \phi(X_0) \psi(X_1) \PN^{\QR} ( F|X_0,X_1 ) \right) .
\end{eqnarray*}
Thus $ \QR (F|X_0,X_1) = \PN^{\QR}(F|X_0,X_1) \ \QR \ps $ for arbitrary functions $F$ which implies that 
\begin{equation*}
\QR(. |X_0=x, X_1=y) = \PNB  \quad \QR_{01} \ps 
\end{equation*}
and the decomposition written in Definition \ref{eq:clasrec} follows.
\end{proof}

%To conclude this section, we proved that any distribution $\QR$ in $\Rec(\nu)$ is an $h$-transform of the coumpound Poisson process with jump measure $\nu$,
%whre $h$ 
%is a function of the initial and final time marginals.

%%%%%%%%%%%%%%%%%%%%%%%%%%%%%%%%%%%%%%%%%%%%%%%%%%%%%%%%%%%%%%%%%
\section{The time and space transformations } \label{sec: trans }
%%%%%%%%%%%%%%%%%%%%%%%%%%%%%%%%%%%%%%%%%%%%%%%%%%%%%%%%%%%%%%%%%
%%%%%%%%%%%%%%%%%%%%%%%%%%%%%%%%%%%%%%%%%%%%%%%%%%%%%%%%%%%%%%%%%
In this section we define two families of transformations on the path space $\OO$, and we analyze their action on the probability measures of the
reciprocal class $\Rec(\nu)$. 
This will later provide (in Theorem \ref{thm:char}) a characterization 
of $\Rec(\nu)$ as the set of probability measures under which a functional equation holds true. 
As a consequence, we will see that $\Rec(\nu)$ depends on $\nu$ only through a family of specific functionals of $\nu$, 
that we call {\em reciprocal invariants}.

%%%%%%%%%%%%%%%%%%%%%%%%%%%%%%%%
\subsection{Time Changes}
%%%%%%%%%%%%%%%%%%%%%%%%%%%%%%%%%
 We consider  the set $\U$ of all regular diffeomorphisms of the time interval $[0,1]$, parametrized by the set $\A$: 
\begin{multline*}
\U = \Big\{ u  \in C^{1} ( \{1,\cdots,A\} \times [0,1];[0,1] ), u(\cdot,0) \equiv 0,u(\cdot,1) \equiv 1,  \\  \min_{j \in A, t \in [0,1]} \dot{u}(j,t) > 0 \Big\}.
\end{multline*}

%\begin{remark}
%Any $u \in \U^{A}$ extends naturally to a diffeomorphisms of $\Delta_{\nv}$ in the following way:
%\begin{eqnarray*}
%\tilde{u} &:& \Delta_{\nv} \longrightarrow \Delta_{\nv}\\
%\tilde{u}(\T^j_i)&:=& u^{j}(\T^j_i), \quad \forall j \leq A,   i \leq n^j \end{eqnarray*} 
% \end{remark}
 
With the help of each $u \in \U$ we construct a transformation of the reference compound Poisson process by time changes acting separately on each
component process $N^j, j=1,...,A$.
\begin{definition}\label{def:timepert}
Let $u \in \U$. We define the time-change transformation $\pi_u$ by: 
\begin{eqnarray*}\label{eq:timepert}
\pi_{u}:  \OO  & \longrightarrow & \mathbb{D}([0,1],\RD) \\
 \pi_{u}(\om)(t) &:= & \om(0) +\sum_{j =1}^{A} a^{j} N^{j}_{u(j,t)}(\om), \ 0\leq t \leq 1.
\end{eqnarray*}
\end{definition}
\begin{remark} 
We cannot a priori be sure that $\pi_u$ takes values in $\Omega$ since
it may happen that jumps sincronize, i.e. $u^{-1}(j,T^j_i) = u^{-1}(j',T^{j'}_{i'})$
 for some $j,j'$. 
 However it is easy to see that this happens with zero probability under $\PNX$.
\end{remark}
%%%%%%%%%%%%%%%%%%%%%%%%%%%%%%%%%%%%%%%%%%%%%%%%%%%%%%%%%%%%%%
We now define a family of maps called {\it reciprocal time-invariants}.
\begin{definition}
 The  \textbf{ reciprocal time-invariant} associated to $\nu$ is the function:
\begin{eqnarray*}
\tinv: \{1,\cdots,A\} \times  [0,1]^2 \rightarrow \R_{+} ,  \quad \tinv (j,s,t) := \frac{\nu^j(t)}{\nu^j(s)} .
\end{eqnarray*}
\end{definition}
\begin{remark}
Note that in the time-homogeneous case $\tinv \equiv 1$.
\end{remark}
%%%%%%%%%%%%%%%
In the next proposition we shall prove that the image of $\PNX$ under the above time change $\pi_u$ is absolutely continuous with respect to $\PNX$,
and that its density is indeed a function of the reciprocal time-invariant  $\tinv$.

%%%%%%%%%%%%%%%%%%%%%%%%%%%%%%%%%%%%%%%%%%%%%%%%%%%%%%%%%%%%%
\begin{prop}\label{prop:udns}
 The following functional equation holds under $\PNX$: For all $u \in \U$, 
 \begin{equation}\label{eq: udns}
  \PNX \Bigl(F \circ \ \pi_{u} \Bigr) =\PNX \Bigl( F \GT \Bigr),  \quad \forall F \in \BM(\Omega)  .
 \end{equation}
\end{prop}
%%%%%%%%%%%%%%%%%%%%%%%%%%%%%%%%%%%%%%%%%%%%%%%%%%%%%%%%%%%%%

\begin{proof}
We first observe that, for every fixed $j \in \{1,...,A\}$ the process
\begin{equation}
N^j _{t} \circ \pi_u - \int_{0}^{t} \nu^j(u(j,s))  \dot{u}(j,s) ds
\end{equation}
is a  $\PNX$-martingale w.r.t. to its natural filtration $\tilde{\mathcal{F}}$.
%$\tilde{\mathcal{F}}_{s} = \sigma( (N^j_r \circ \pi_{u})_{r \in [0,s]} ) $.
Indeed, for any $s\leq t$ and any $F$ $\mathcal{\tilde{F}}_s$-measurable,  by applying the basic properties of
processes with independent increments, we obtain:
\begin{equation*}
\begin{split}
\PNX \big( F \ (N^j_t-N^j_s) \circ \pi_u  \big) &  =  
\PNX \big(F \big) \int^{u(j,t) }_{u(j,s)} \nu^j(r) dr \\ & = \PNX \big(F \big)\int^{t }_{s}\nu^j (u(j,r)) \dot{u}(j,r) dr .
\end{split}
\end{equation*}
Therefore $N^j _{t} \circ \pi_u$ is an inhomogeneous Poisson process with intensity  $$\nu^j (u(j,\cdot)) \dot{u}(j,\cdot).$$
Moreover, if $j \neq j'$, $N^j_\cdot \circ \pi_u$ and $ N^{j'}_\cdot \circ \pi_u $ are independent processes under $\PNX$, because the
processes $N^j$ and $N^{j'}$ are independent and $\pi_u$ acts separately on each component.
This implies that the image of $\PNX$ under  $\pi_u $, $\PNX \circ \pi^{-1}_u$, is a compound Poisson process whose jump measure is given by
$ \sum_{j=1}^A  \delta_{a^j} (dx)\otimes \nu^j (u(j,t)) \dot{u}(j,t) dt $.\\
We can now apply the Girsanov theorem (see e.g. Theorem 5.35, Chapter III of
\cite{JaShi03}) to get the density of the
push-forward measure $\PNX \circ \pi^{-1}_u$ w.r.t. $\PNX$:
\begin{multline*}
\frac{d\PNX \circ \pi_u^{-1}}{d \PNX} = \exp \Big[ \sum_{j=1}^{ A} \Big( \int_0^1 \big( \nu^j (u(j,t)) \dot{u}(j,t)- \nu^j(t)\big) dt  \\ +  
\int_{0}^{1} \log(\Xi^{\nu}(j,t,u(j,t))\dot{u}(j,t) dN_t^j \Big) \Big].
\end{multline*}
With the change of variable $t=u^{-1}(t')$ we have for any $j$
\begin{equation*}
\int_{0}^{1}  \nu^j (u(j,t)) \dot{u}(j,t) dt  =\int_{0}^{1} \nu^j(t')dt' .
\end{equation*}
Therefore the first integral disappears and the conclusion follows.
\end{proof}

\begin{remark}
In the recent work \cite{LMR},
%, Theorem 2.10)
the authors establish a differential version of the equation \eqref{eq: udns}, in the context of counting processes 
(i.e. $\A= \left\{ 1 \right\}$) with a possibly space-dependent intensity.
Such a formula is inspired by the Malliavin calculus for jump processes 
developed in \cite{carlen88} puting in
duality a differential operator
%, acting on smooth functionals of the jump times,
and the stochastic integral.
We can, without getting into the details, relate the functional equation \eqref{eq: udns} with the formula proved there as follows:
First consider a smooth function $v$ satisfying the loop condition $\int_{0}^{1} v_t \ dt = 0 $ and 
define the function $u^{\varepsilon}_t:=
t+ \varepsilon \int_0^t v_s \ ds $, $\varepsilon>0$. Note that, for $\varepsilon$ sufficiently small, $u^{\varepsilon} \in \U$. 
If we then apply \eqref{eq: udns} to a smooth
functional of the jump times, we obtain after some elementary operations:
\begin{equation*}
\frac{1}{\varepsilon} \PNX ( F \circ \pi_{u^{\varepsilon}} - F ) = 
\frac{1}{\varepsilon}\PNX \bigg( F \Big( \exp ( \int_{ [0,1]} \log \Xi^{\nu}(1,t,u^{\varepsilon}(t)) \ \dot{u}^{\varepsilon}_t \  dN_t ) -1 \Big) \bigg) .
\end{equation*}
If we now let $\varepsilon $ tend to 0, we obtain the duality formula
\begin{equation*}
\PNX \big( D_v F \big) = \PNX \big( F \int_{0}^{1} v_t \ dN_t )+ \PNX \big( F \int_{0}^{1} v_t \int_{t}^{1}  \frac{\dot{\nu}(s)}{\nu(s)} dN_s dt \big)
\end{equation*}
where $
D_v F = \lim_{\varepsilon \rightarrow 0} (F \circ \pi_{u^{\varepsilon}} - F)/\varepsilon. $\\
This formula can then be extended to space-dependent intensities.
\end{remark}
%%%%%%%%%%%%%%%%%%%%%%%%%%%%%%%%%%%%%%%%%%%%%%%%%%%%%%%%%%%%%%%%%%%%%%%%%%%%%%%%%%%%%%%%%%%%%%%%%%%%%%%%%%%%%%%%%%%%%%%%%%%%%%%%%%%%%%%%%%%%%%%%%%%%%%%%%%%%%%%%%%%%%%%%%%%%%%%%%%%%%%%%%%%%%%%%%

\subsection{Space transformations}
The transformations $\pi_u$ introduced in the previous section, when acting on a given path, 
change the jump times leaving unchanged the total number of jumps of each type. We now introduce transformations that modify the total
number
of jumps; these transformations act on the counting variable $\mathbf{N}_1 $  taking its values in $ \N^{A}$, which we embed into $\Z^{A}$ to take
advantage of the lattice structure.

\subsubsection{Shifting a Poisson random vector}
 We first consider  a
multivariate Poisson distribution ${\mathfrak{p}}_{\lambda} \in \PM(\N^A)$ where$\lambda=(\lambda^1,...,\lambda^A) \in \R_{+}^A$:
\begin{equation}\label{mpoisson}
\forall \nv \in \N^A, \quad \pl (\nv) = \exp \left(- \sum_{j=1}^{A} \lambda^j \right) \prod_{j=1}^{A} \frac{(\lambda^j)^{n^j}}{n^j!} \quad .
\end{equation}

We first give a straightforward multidimensional version of Chen's characterization of
a Poisson random variable. He 
 introduced it  to estimate the rate of convergence of
sum of dependent trials to the Poisson distribution 
(see the original paper \cite{Chen} and Chapter 9 in \cite{Ste86} for a complete account of Chen's
method). 
\begin{prop}\label{canonicalChen}
Let $\lambda \in (\R_{+})^A$. Then $\rho \in \mathcal{P}(\N^A)$ is the multivariate Poisson distribution  $ \pl$ if and only if 
\begin{equation*}
 \forall \mathbf{e}^j, j=1, \dots A, \quad \rho( f ( \nv + \mathbf{e}^j))  =\lambda^j \rho ( f(\nv) n^j), \quad \forall f \in  \BM(\N^A),
\end{equation*}
where $\mathbf{e}^j$ denote the $j$-th vector of the canonical basis of $\Z^A$. 
\end{prop}

One can interpret this characterization as the computation of the density of the image measure by any 
shift along the canonical basis of $\N^A$. 
%It will be proven in more generality later on.

Now we consider  as more general transformations multiple left- and right-shifts, acting simultaneously on each coordinate, that is, we shift by vectors $\vv 
\in \Z^A$. 

%%%%%%%%%%%%%%%%%%%%%%%%%%%%%%%%%%%%%%%%%%%%%%%%%%%%%%%%%%%%%%%%%%%%%%%%%%%%%%%%%%%%%%%%%%%%%%%%%%%%%%%%%%%%%%%%%%%%%%%%%%%%%%%%%%
\begin{definition}
Let  $\vv \in \Z^{A}$. We define the $\vv$-shift by 
\begin{eqnarray*}\label{def:spaceshift}
\shv :\Z^{A} &\longrightarrow& \Z^A \\
  \zv &\mapsto& \shv(\zv ) = \zv +\vv  .
 \end{eqnarray*}
\end{definition}
%%%%%%%%%%%%%%%%%%%%%%%%%%%%%%%%%%%%%%%%%%%%%%%%%%%%%%%%%%%%%%%%%%%%%%%%%%%%%%%%%%%%%%%%%%%%%%%%%%%%%%%%%%%%%%%%%%%%%%%%%%%%%%%%%%%%%%%%%%%%%%%%%%%%%%%%%%%%%%%%
%%%%%%%%%%%%%%%%%%%%%%%%%
Consider  the image  of $\pl$ under $\theta_{\vv}$. It is a probability measure whose support is no more included in $\N^A$ since  there may be $\zv \in \N^A$
such that $\shv(\zv)
\not\in \N^A$. 
Therefore  we only compute the density of its absolutely continuous component, appearing in 
the Radon-Nykodim decomposition:
\begin{equation}\label{RadNik}
\pl \circ \shv^{-1} = \pl^{\vv,ac} + \pl^{\vv,sing}.
\end{equation}
A version of the density of the absolutely continuous component is given by
\begin{equation*}
\frac{d\pl^{\vv,ac}}{d \pl } (\nv)=\lambda^{-\vv} \prod_{j=1}^{A}\frac{n^j!}{(n^j-v^j)!} \mathds{1}_{\left\{ n^j  \geq  v^j  \right\}}, \textrm{ where } 
\lambda^{\vv}:=\prod_{j=1}^A (\lambda^j)^{v^j}. 
\end{equation*}
In view of obtaining a change of measure formula as in Proposition \ref{canonicalChen} we define
\begin{equation} \label{Gc}
\GC(\nv):= \prod_{j=1}^{A}\frac{n^j!}{(n^j-v^j)!} \mathds{1}_{\left\{ n^j  \geq  v^j  \right\}} .
\end{equation}
Let us now consider the space $\BM^{\sharp}(\Z^A) \subseteq \BM(\Z^A)$ consisting of test functions with support in $\N^A$:
\begin{equation*}
\BM^{\sharp}(\Z^A):=\{ f \in \BM(\Z^A): f(\zv) = 0 \ \forall \zv \notin \N^A \} .
\end{equation*}
Then, the considerations above can be summarized in the following formula:
\begin{equation}\label{vshift}
  \pl( f \circ \shv )  =\lambda^{-\vv} \ \pl ( f \  \GC), \quad \forall f \in  \BM^{\sharp}(\Z^A) .
\end{equation}

\begin{example}\label{ex:shift}
Let $\mathcal{A} = \{ -1,1 \}$. We call $n^-$ (rather than $n^1$) and $n^+$(rather than $n^2$) the counting variables for the jumps $-1$ and $1$ respectively.
The same convention is adopted for the intensity vector $\lambda = ( \lambda^{-} , \lambda^{+} )$. 
Then, for $\vv =(1,1)$ (resp. $\vv= (1,-1)$ and for any $f \in \BM^{\sharp}(\Z^2)$,
\begin{equation*}
 \pl \Bigl( f (\nmin+1,\npiu+1) \Bigr)  = \frac{1}{\lambda^{-} \lambda^{+}} \, \pl \Bigl( f (\nmin,\npiu)
\nmin\npiu  \Bigr),
\end{equation*}
\begin{equation*}
 \pl \Bigl( f (\nmin+1,\npiu-1) \Bigr)  = \frac{\lambda^{+}}{ \lambda^{-}} \, \pl \Bigl( f (\nmin,{\npiu})
\frac{\nmin}{\npiu+1}  \Bigr).
\end{equation*}
\end{example}
%%%%%%%%%%%%%%%%%%%%%%%%%%%%%%%%%%%%%%%%%%%%%%%%%%%%%%%%%%%%%%%%%%%%%%%%%%%%%%%%%%%%%%%%%%%%%%%%%%%%%%%%%%%%%%%%%%%%%%%%%%%%%%%%
\subsubsection{Lattices and conditional distributions}

We now consider, associated to a measure $\mu \in \PM(\N^A)$, the following set of probability measures on $\N^A$:
%, discrete analogous of the reciprocal class:
\begin{equation*}
\drec:= \big\{ \rho \in \PM(\N^A):  \rho(\cdot)=\int \mu(\cdot | \sigma(\mathbf{A})) \, d\rho_{\sigma(\mathbf{A})} \big\},
\end{equation*}
where the $\sigma$-algebra $ \sigma(\mathbf{A})$ is generated by the application $\zv \mapsto \mathbf{A}\zv$ defined on $\Z^A$, and the measure
$\rho_{\sigma(\mathbf{A})}$ is the projection of $\rho$ on $\sigma(\mathbf{A})$.\\
The set $\drec$ presents  strong analogies with a reciprocal class as introduced in Definition \ref{def-clasrec}. 
Indeed, one can prove an analogous to Proposition \ref{prop:Q=hP}, that is  
$$\rho  \in \drec \textrm{ if and only if } \rho << \mu   \textrm{ and } \frac{d\rho}{d\mu}  \textrm{ is } \ \sigma(\mathbf{A}) \textrm{-measurable}.$$

%%%%%%%%%%%%%%%%%%%%%%%%%%%%%%%%%%%%%%%%%%%%%%%%%%%%%%%%%%%%%%%%%%%%%%%%%%%%%%%%%%%%%%%%%%%%%%%%%%%%%%%%%%%%%%%%%%%%%%%%%%%%%%%%%%%%%%%%%%%%%%%%%%%%%
 Our first goal is to characterize $\drecl$ using the formula \eqref{vshift} computed for a suitably chosen set of shift vectors $\vv$. 
 The right set will be the following sublattice of $\Z^A$:
 \begin{equation}\label{def:C}
 \ker_\Z(\mathbf{A}):= \ker(\mathbf{A}) \cap \Z^A .
 \end{equation}
Let us observe that if two paths $\om, \tilde{\om} \in \Omega $ have the same initial and final values, $(X_0,X_1)( \om ) = (X_0,X_1)( \tilde{\om})$, then
$\NO(\om) - \NO(\tilde{\om}) \in \ker_\Z(\mathbf{A}) $.
The next statement clarifies the role of $\C$. 

% \begin{equation*}
% \BM^{\sharp}(\Z^A):= \left\{ f \in \BM(\Z^A ) : \  f(\zv) = 0  \  \forall \zv \in \Z^A \setminus \N^A \right\}
% \end{equation*}
%%%%%%%%%%%%%%%%%%%%%%%%%%%%%%%%%%%%%%%%%%%%%%%%%%%%%%%%%%%%%%%%%%%%%%%%%%%%%%%%%%%%%%%%%%%%%%%%%%%%%%%%%%%%%%%%%%%%%%%%%%%%%%%%%%%%%%%%%%%%%%%%%%%%%%%%%%%%%%%%%%%%%%
\begin{prop}\label{prop:cycleshift}
Let $\rho \in \mathcal{P}(\N^A)$. Then $\rho \in \drecl$ if and only if 
\begin{equation}\label{eq:cycleshift}
  \forall \cv \in \C, \quad \rho( f \circ \theta_{\cv})  = \frac{1}{\lambda^{\cv}} \, \rho ( f \ \GCC ) \quad \forall f \in  \BM^{\sharp}(\Z^A) , 
\end{equation}
where $ \GCC$ is defined in \eqref{Gc}.
\end{prop}
\begin{proof}
 ($\Rightarrow $)\,  Let $f \in \BM^{\sharp}(\Z^A)$ and $\cv \in \C$. By definition of $\C$ and $\drecl $ we can choose a version of the
density $h=\frac{d\rho}{d\pl}$ such that $h \circ \shc = h $. Applying the formula \eqref{vshift} we obtain:
\begin{eqnarray*}
\rho( f \circ \shc )  &=& \pl\left( (f \circ \shc) h\right) = \rho \left( (fh) \circ \shc  \right)\\
&=& \lambda^{-\cv} \pl \left( f \ \GCC  h\  \right) =  \lambda^{-\cv} \rho \left(  f \ \GCC\right)  
\end{eqnarray*}
($ \Leftarrow$) \, Let $\nv,\mathbf{m} \in \N^A$ be such that $\mathbf{A} \nv = \mathbf{A} \mathbf{m} $. Set $f := \mathds{1}_{\mathbf{n}}$, $\cv :=
\mathbf{n}- \mathbf{m}$. Then
\[
\rho(\mathbf{m}) = \rho(f \circ \shc) = \lambda^{-\cv} \GCC(\mathbf{n}) \rho(\mathbf{n}).
\]
Since, by \eqref{vshift}, the same relation holds under $\pl$, we have
\[
\frac{d\rho}{d\pl}(\mathbf{m}) = \frac{d\rho}{d\pl}(\mathbf{n}),
\]
which completes the proof.
\end{proof} 
%%%%%%%%%%%%%%%%%%%%%%%%%%%%%%%%%%%%%%%%%%%%%%%%%%%%%%%%%%%%%%%%%%%%%%%%%%%%%%%%%%%%%%%%%%%%%%%%%%%%%%%%%%%%%%%%%%%%%%%%%%%%%%%%%%%%%%%%%%%%%%%%%%%%%%%%%%%%%%%%
\begin{example}\label{ex:dumbshift}
Resuming Example \ref{ex:shift}, we verify that, in this case,
%\begin{equation*}
$\C  =\left(
\begin{array}{c}
1\\
1\\
\end{array}
\right)  \Z .
$
%\end{equation*}
Proposition \ref{prop:cycleshift} tells us that a  probability distribution $\rho$ on $\N^2$ satisfies
\begin{equation*}
 \rho(. \ |\npiu - \nmin = x) = \pl(. \ |\npiu - \nmin = x ) \ \ \forall x \in \Z
\end{equation*}
if and only if, for all $k$ in $\N^*$ and for all $f \in \BM^{\sharp}(\Z^2)$,
\begin{equation*} \label{Poik}
 \rho \Bigl( f (\nmin+k,\npiu+k) \Bigr)  = \frac{1}{(\lambda^{+} \lambda^{-})^{k}} \  \rho \Bigl( f (\nmin,\npiu) \prod_{i=0}^{k-1}(\nmin-i)(\npiu -i) \Bigr)
\end{equation*}
and
\begin{equation*} \label{Poik-}
 \rho \Bigl( f ( \nmin-k, \npiu-k ) \Bigr)  = (\lambda^{+} \lambda^{-})^{k} \  \rho \Bigl( f (\nmin,\npiu) \prod_{i=1}^{k}\frac{1}{(\nmin+i)(\npiu + i)} \Bigr).
\end{equation*}
\end{example}
\vspace{3mm}

Taking Proposition \ref{prop:cycleshift} as a characterization of $\drecl$  is rather unsatisfactory, since we do not exploit the lattice structure of $\C$.
It is then natural to look for improvements, by imposing \eqref{eq:cycleshift} to be satisfied only for shifts in a smaller subset of  $\C$, but still
characterizing $\drecl$. 
In particular, since $ \C $ is a sublattice of $\Z^A$, one wants to understand  if restricting to a basis suffices. 
 We answer affirmatively if $\C$ satisfy a
simple geometrical condition, see Proposition \ref{posray}. However, in general this is false, and we construct the Counterexample \ref{counterex}. \\
%%%%%%%%%%%%%%%%%%%%%%%%%%%%%%%%%%%%%%%%%%%%%%%%%%%%%%%%%%%%%%
Before doing so, let us recall that  $\mathcal{L} \subseteq \R^A$ is a \textit{lattice} if there is a set of
linearly independent vectors $\{ b_1,..,b_l \}$ such that:
\begin{equation}\label{def:basis}
\mathcal{L} = \Big\{ \sum_{k=1}^{l} z_k b_k , \ z_k \in \Z, \  k =1,..,l \Big\}.
\end{equation}
Any set $\{z_1,...,z_l\}$ satisfying \eqref{def:basis} is a \textit{basis} for $\mathcal{L}$. Since any discrete subgroup of $\R^A$ is a
lattice (see e.g. Proposition 4.2 in \cite{NK}) then $\C$ is a lattice.\\
%%%%%%%%%%%%%%%%%%%%%%%%%%%%%%%%%%%%%%%%%%%%%%%%%%%%%%%%%%%%
The equations \eqref{eq:cycleshift} essentially tell us that, if $\nv \in \N^A$ is such that $ \mv:=\theta_{-\cv} \nv$ is also an element of $\N^A$, then
$\rho(\mv ) = \rho(\nv) \pl(\mv)/\pl(\nv)$. If we let $\cv$ vary in a lattice basis it may happen that the ''network'' of relations constructed in this
way is not enough to capture the structure of  conditional probabilities. \ In the next paragraph, we will indeed reformulate this problem as a connectivity
problem for a certain family of graphs, and propose a solution in this framework using \textit{generating sets} of lattices. 
%%%%%%%%%%%%%%%%%%%%%%%%%%%%%%%%%%%%%%%%%%%%%%%%%%%%%%%%%%%%%%%%%%%%%%%%%%%%%%%%%%%%%%%%%%%%%%%%%%%%%%%%%%%%%%%%%%%%%%%%%%%%%%%%%%%%%
\begin{counterex}\label{counterex}
Let $\A = \left\{ 3,4,5 \right\}$. Then $$\C= \left\{ \cv \in \Z^3 : 3 c_1 + 4 c_2 + 5 c_3 =0 \right\}.$$ 
We define three vectors
\begin{equation*}
f=(-3,1,1) , \quad  g=(1,-2,1) , \quad h = (2,1,-2).
\end{equation*}
Note that $\{ f,g,h \} \subseteq \C$. We also define
\begin{equation*} 
\nv_f: = (3,0,0), \quad   \nv_g:= (0,2,0), \quad  \nv_h:= (0,0,2).
\end{equation*}
Moreover, we observe that 
\begin{equation}\label{2vertex}
\mbox{\textit{if, for some $\cv \in \C, \, \shc \nv_f \in \N^3 $ then $\cv=f$}}.
\end{equation}
This can be checked with a direct computation.
The analogous statement also holds for $g$ and $h$, i.e.  
$$
\shc \nv_g \in \N^3 \Rightarrow \cv= g,\quad \shc \nv_h \in \N^3 \Rightarrow \cv= h.
$$
Let us now consider any basis $\CST$ of $\C$. 
Since $\C$ is two dimensional, at least one vector, $f$ or $g$ or $h$, does not belong to $\CST$. We assume w.l.o.g that $f
\notin \CST$.
For any $0<\varepsilon < 1$, $\lambda \in \R^3_+$, we define the probability measure $\rho \in \PM(\N^3)$ as a mixture between the degenerate measure  
$\delta_{\nv_f}$ and $\pl$ as follows:
\begin{equation}\label{romix}
\rho = \varepsilon \delta_{\nv_f} + (1-\varepsilon) \pl .
\end{equation}
Note that  $\rho \notin \drecl $. Indeed any version of the density must be such that:
 \begin{equation*}
\frac{d \rho }{ d\pl} (\nv_f)= \frac{\varepsilon}{\pl(\nv_f)}+(1-\varepsilon), \quad \frac{d \rho }{ d\pl} (\theta_{\cv_f}\nv_f)= 1-\varepsilon .
\end{equation*}
But, on the other hand, identity \eqref{eq:cycleshift} is satisfied for any  $\cv \in \CST$. 
Let us pick any test function $f = \mathds{1}_{\{\zv = \bar{\nv}\}}$, where $\bar{\nv} \in \N^3$ and $\cv \in \CST$.
There are two possibilities:\\
- Either $\theta_{-\cv} \bar{\nv} \in \Z^3 \setminus \N^3$.
In this case \eqref{eq:cycleshift} is satisfied by $\rho$ because both sides of the equality are zero, 
the left side because $\theta_{-\cv} \bar{\nv} \notin \N^A, \rho(\N^A)=1$ and the right side because $\GCC(\bar{\nv})=0$. \\
%(Recall that $\GCC$ has been definedin \eqref{Poidns})%
- Or $\theta_{-\cv} \bar{\nv} \in \N^A$. In this case, thanks to \eqref{2vertex} and $f \notin \CST$ we have  
$\bar{\nv} \neq \nv_f$ and $\theta_{-\cv} \bar{\nv} \neq \nv_f$. Therefore, by \eqref{romix},
\begin{equation*}
\begin{split}
\rho(\mathds{1}_{\{ \shc \zv = \bar{\nv} \}}  ) & = \frac{\rho(\theta_{-\cv}\bar{\nv})}{\rho(\bar{\nv}) } \rho(  \mathds{1}_{ \left\{  \zv = \bar{\nv} \right\} } )=
\frac{\pl(\theta_{-\cv}\bar{\nv})}{\pl(\bar{\nv}) } \rho(  \mathds{1}_{ \left\{  \zv = \bar{\nv} \right\} } ) \\ & =   \lambda^{-\cv} \rho(  \mathds{1}_{ \left\{ \zv = \bar{\nv} \right\} } \GCC(\zv)) 
\end{split}
\end{equation*}
which is equivalent to  \eqref{eq:cycleshift}.\\
We thus obtain an example of a set $\A$ such that, for any $\lambda \in \R^3_{+}$ and  
any basis $\CST $ of $\C$ we can construct a probability measure $\rho$
which satisfies \eqref{eq:cycleshift} for $\cv \in \CST $ and $  f \in \BM^{\sharp}(\Z^A)$  but does not belong to $\drecl$.
\end{counterex}
%%%%%%%%%%%%%%%%%%%%%%%%%%%%%%%%%%%%%%%%%%%%%%%%%%%%%%%%%%%%%%%%%%%%%%%%%%%%%%%%%%%%%%%%%%%%%%%%%%%%%%%%%%%%%%%%%%%%%%%%%%%%%%%%%%%%%%

%%%%%%%%%%%%%%%%%%%%%%%%%%%%%%%%%%%%%%%%%%%%%%%%%%%%%%%%%%%%%%%%%%%%%%%%%%%%%%%%%%%%%%%%%%%%%%%%%%%%%%%%%%%%%%%%%%%%%%%%%%%%%%%%
\subsubsection{Generating sets of lattices and conditional distributions}

%We recall that a lattice is a subset which is a discrete subgroup of $\R^A$.
%With this definition, it is clear that $\C$ defined in \eqref{def:C} is a lattice. 
We  first  define the {\em foliation} that the lattice $\C$ induces on $\N^A$: given $\nv \in \N^A$, let define the leaf containing $\nv$ by 
\begin{equation}\label{foliation}
\LF:= \left\{\nv + \C \right\}  \cap \N^A .
\end{equation}
Fix now $\CS \subseteq \C$. $\CS$ induces a graph structure on each leaf (see e.g. \cite{NK}):

\begin{definition}
 For $\CS \subseteq \C$ and $ \nv \in \N^A$ we define  $\mathcal{G}(\LF,\CS)$ as the graph whose vertex  set is $\LF$ 
 and whose edge set is given by
$$ 
 \left\{ (\mathbf{m},\mathbf{m}') \in \LF \times \LF : \exists \  \cv \in \CS \ \textrm{ with } \  \mathbf{m} = \shc(\mathbf{m}') \right\} .
$$
\end{definition}
%In the next definition we introduce the natural "foliation" that a lattice $\C$ induces on $\N^A$ and the graph that a given $\CS \subseteq \C$ induces on each leaf: two points are joined by an edge if and only one can move from one to the other by applying a vector in $\CS$ (cf Definition 11.1 in \cite{Koppe13}).
%\begin{definition}\label{graph}
%Let $\C $ be a lattice, and $\nv \in \N^A$. We define:
%\begin{equation*}\label{foliation}
%\mathfrak{F}_{\C,\nv}:= \left\{\nv + \C \right\}  \cap \N^A 
%\end{equation*}
%For $\nv \in \N^A$, $\CS \subseteq \C $ we define the undirected graph $\mathcal{G}(\mathfrak{F}_{\C,\nv},\CS)$ whose vertex set is $\mathfrak{F}_{\C,\nv}$ and whose edge set is: 
%\begin{equation*}
% \left\{ (\mathbf{m},\mathbf{m}') \in \LF \times \LF : \exists \cv \in \CS s.t. \mathbf{m} = \shc(\mathbf{m}') \right\}
%\end{equation*}
%
%\end{definition}
We are now ready to  introduce the notion of generating set for $\C$.
\begin{definition}\label{def: genset}
The set $\CS$ is a {\em generating set} for $\C$ if, for all $\nv \in \N^A$, 
$\mathcal{G}(\LF,\CS)$ is a connected graph for the graph structure we have just defined.
%\begin{itemize}
%\item[i)]For $\nv \in \N^A $ we say $\CS$ a \textit{generating set} for $\LF$ if the graph $\mathcal{G}(\mathfrak{F}_{\C,\nv},\CS)$ is connected.
%\item[ii)]  $\CS \subseteq \C$ a \textit{generating set} for $\C$ if it is a generating set for $\LF$, $\forall \nv \in \N^A$.
%\end{itemize}
\end{definition}

 We now recall, following Chapters 10 and 11 of the recent book \cite{Koppe13} - in which an extensive study of generating sets for lattices is done -  
that there  exists  a finite generating set for any given lattice.  
% 
% \begin{thm}\label{thm:genset}
% Let $\C$ be a lattice. Then there exists a  finite generating set $\CS$ for $\C$.
% \end{thm}
Finding a generating set for a lattice $\C$ is indeed closely related to the algebraic problem of finding a set of generators for the \textit{lattice ideal}
associated to $\C \subseteq \Z^A$, see Lemma 11.3.3 in \cite{Koppe13}.\\

Any generating set contains a lattice basis, but in general it might be much larger. 
Figure 1 illustrates a case where $\{ \left(  2,  -4,  2  \right),\left(  0 , -5 , 4  \right) \}$, basis of $\C$, is not a generating set 
since the graph $\mathcal{G}(\LF,\CS)$ is not connected for $\nv =(6,  1 , 2)$.\\

Computing  explicitly generating sets is a hard task. We give here two simple conditions on $\C$ under which a lattice basis is also a generating set.
The proof is given in the Appendix.
\begin{figure}[h!]
\includegraphics[width=\linewidth]{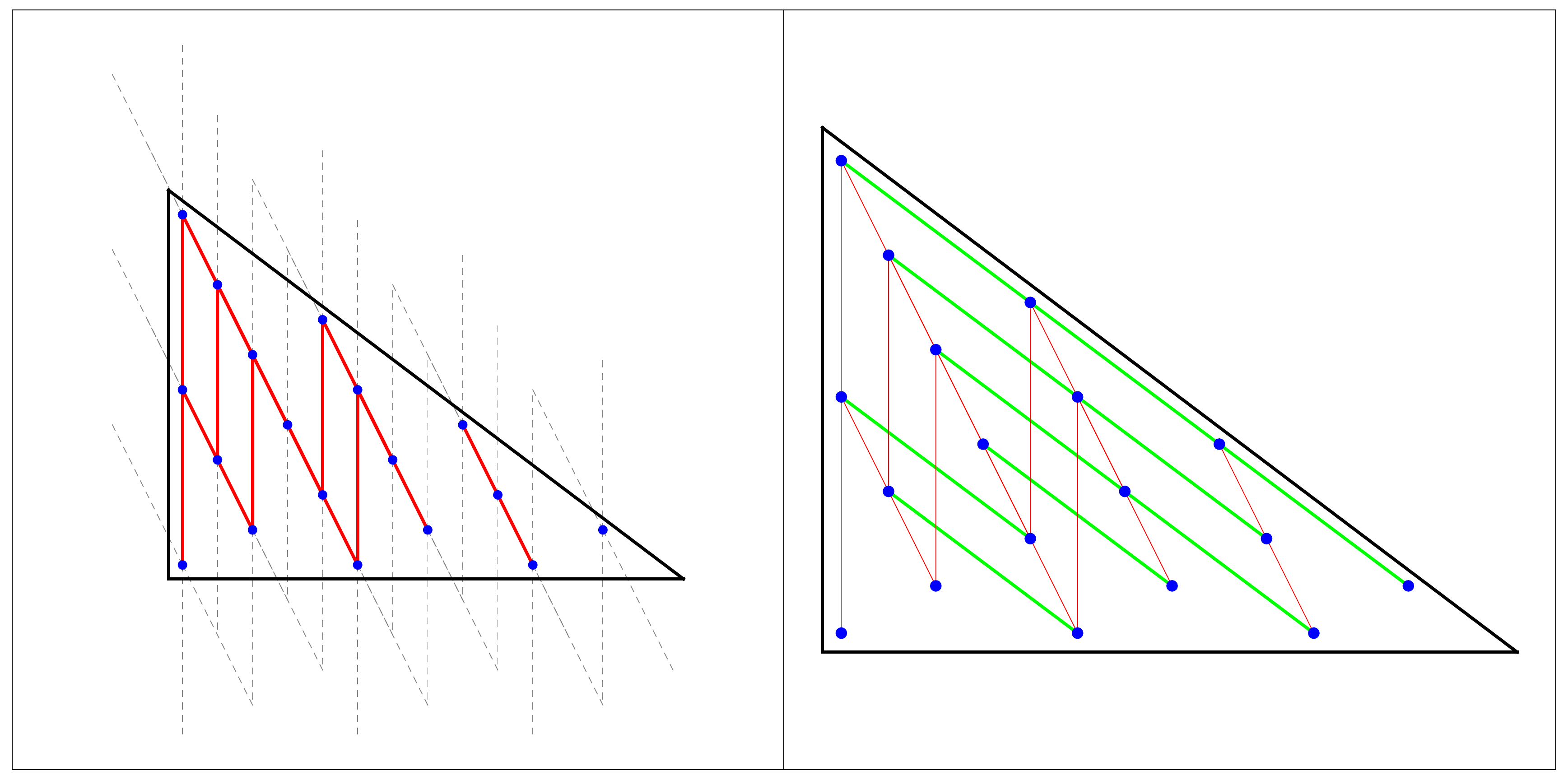}
\caption{$\A = \{ 3,4,5 \} $ and $\CST=\{ \left(  2,  -4,  2  \right),\left(  0 , -5 , 4  \right) \}$. 
Left: Projection on the $x_1 x_2 $ plane
of $G:=\mathcal{G}(\LF,\CST)$ for  $\nv =(6,  1 , 2)$. 
The red lines are the edges of $G$, while the dashed lines
represent edges that are not in $G$ because one endpoints does not belong to $\N^3$.
The graph $\mathcal{G}(\LF,\CST)$ has three connected components.
Right: Adding the vector $(4,-3,0) $ to $\CST$ turns $G$ into a connected graph.}
\end{figure}
%\vspace{0.5cm}
\begin{prop}\label{posray}
Let $\CST$ be a basis of $\C$. Suppose that one of the following conditions holds:\\
$i$) The basis  $\CST$  contains an element $\bar{\cv}$ such that each coordinate  $\bar{c}^j, j=1,...,A$ is positive. \\
$ii$) Each vector of the basis  $\CST$ is an element of $ \N^A$.\\
Then, the basis $\CST$ is a generating set.
\end{prop}

%%%%%%%%%% .%%%%%%%%%%%%%%%%%%%%%%%%%%%%%%%%%%%%%%%%%%%%%%%%%%%%%%%%%%%%%%%%%%%%%%%%%%%%%%%%%%%%%%%%%%%%%%%%%%%%%%%%%%%%
In the next theorem we show how one can use generating sets to characterize the set of probability measures $\drec$. 
Even though we are interested here in the case $\mu = \pl$ the statement is proven in a slightly more general context. 
In the same spirit as in \eqref{RadNik},
we consider the Radon-Nykodim decomposition of the image of $\mu$ by  $\shc$, $\mu \circ \shc^{-1} = \mu_{\cv}^{ac} + \mu^{sing}_{\cv}$, and the density
of $\mu^{ac}_{\cv}$ with respect to $\mu$:
%%%%%%%%%%%%%%%%%%%%%%%%%%%%%%%%%%%%%%%%%%%%%%%%%%%%%%%%%%%%%%%%%%%%%%%%%%%%%%%%%%%%%%%%%%%%%%%%%%%%%%%%%%%%%%%%%%%%%%%%%%%%%%%%%%%%%%
\begin{equation}\label{RadNik2}
\quad G_{\cv}^{\mu} := \frac{d \mu^{ac}_{\cv} }{d \mu} .
\end{equation}
%%%%%%%%%%%%%%%%%%%%%%%%%%%%%%%%%%%%%%%%%%%%%%%%%%%%%%%%%%%%%%%%%%%%%%%%%%%%%%%%%%%%%%%%%%%%%%%%%%%%%%%%%%%%%%%%%%%%%%%%%%%%%%%%%%%%%%%%%%%%%%%%%%%%%%%%%%%%
\begin{thm}\label{latinv}
Let   $\mathbf{A} \in \R^{d \times A}$ be any matrix and 
the lattice  $\C$ be defined as before by
$ \C:= \ker(\mathbf{A}) \cap \Z^A$ .
Assume that $\CS$ is a generating set of $\C$ and
let $\mu , \rho $ be two probability measures on $\N^A$. Suppose moreover that $\mu(\nv)>0$ for all $\nv \in \N^A$.\\
Then $ \rho \in \drec $ if and only if
\begin{equation}\label{hyp}
\forall \cv  \in \CS, \quad \rho ( f \circ \shc  ) = \rho (   f  \ G^{\mu}_{\cv} ) \quad \forall f \in \BM^{\sharp}(\Z^A),
 \end{equation}
 where $G^{\mu}_{\cv}$ is defined by \eqref{RadNik2}.
\end{thm}
%%%%%%%%%%%%%%%%%%%%%%%%%%%%%%%%%%%%%%%%%%%%%%%%%%%%%%%%%%%%%%%%%%%%%%%%%%%%%%%%%%%%%%%%%%%%%%%%%%%%%%%%%%%%%%%%%%%%%%%%%%%%%%%%%%%%%%%%%%%%%%%
\begin{proof}
($\Rightarrow $) goes along the same lines of Proposition \ref{prop:cycleshift} since $\CS \subseteq \C$.\\

%%%%%%%%%%%%%%%%%%%%%%%%%%%%%%%%%%%%%%%%%%%%%%%%%%%%%%%%%%%%%%%%%%%%%%%%%%%%%%%%%%%%%%%%%%%%%%%%%%%%%%%%%%%%%%%%%%%%%%%%%%%%%%%
($\Leftarrow$) 
\, Let $\nv,\mv \in \N^A$ be such that $\mathbf{A} \nv = \mathbf{A} \mathbf{m} $ and assume that $\rho(\nv)>0 $. 
Then $\mathbf{m} \in \LF$ (see \eqref{foliation}). 
Since $\CS$ is a generating set for $\C$ there exists a path from $\mathbf{m}$ to $\nv$ included in $\mathcal{G}(\LF,\CS)$ i.e. 
there exists $\cv_1,...,\cv_K \in \CS$ such that, if we define recursively:
\begin{equation*}
w_0 = \mathbf{m}, \quad w_k = \theta_{\cv_k} w_{k-1},
\end{equation*}
then $w_k \in \N^A$ for all $k $ and $w_K=\nv$. \ We can choose $f^k= \mathds{1}_{\left\{ \zv = w_k  \right\} }$ 
and apply \eqref{hyp} for $\cv = \cv_k$:
\begin{equation*}
\rho(w_{k-1}) = \frac{\mu(w_{k-1})}{\mu(w_k)} \rho(w_{k}) 
\end{equation*}
which, since $\mu$ is a positive probability on $\N^A$, offers   an inductive proof that $\rho(w_{k})>0$. 
Therefore one obtains
$$
\frac{\rho(\mathbf{m})}{\rho(\nv)}  = \prod_{k=1}^{K}\frac{\rho(w_{k-1})}{\rho(w_{k})} 
= \prod_{k=1}^{K}\frac{\mu(w_{k-1})}{\mu(w_{k})} = \frac{\mu(\mathbf{m})}{\mu(\nv)}
$$
which is equivalent to
$d \rho / d \mu \ (\nv) = d\rho/ d \mu \ (\mathbf{m}),
$ 
which completes the proof.
\end{proof}

%%%%%%%%%%%%%%%%%%%%%%%%%%%%%%%%%%%%%%%%%%%%%%%%%%%%%%%%%%%%%%%%%%%%%%%%%%%%%%%
As consequence of Theorem \ref{latinv}, we obtain the following statement, which improves Proposition \ref{prop:cycleshift}.
%%%%%%%%%%%%%%%%%%%%%%%%%%%%%%%%%%%%%%%%%%%%%%%%%%%%%%%%%%%%%%%%%%%%%%%%%%%
\begin{cor}\label{prop:GenChen}
Let $ \rho \in \mathcal{P}(\N^A) $ and $\CS$ be a generating set of $\C$ defined by \eqref{def:C}.
Then  $\rho \in \drecl $ if and only if
\begin{equation}\label{csetChen}
  \forall \cv \in \CS, \quad \rho \left( f \circ \shc \right) =
  \frac{1}{\lambda^{\cv}} \rho \left( f  \ \GCC \right),\quad \forall f \in  \BM^{\sharp}(\Z^A),
\end{equation}
where $\GCC$ is defined in \eqref{Gc}.
\end{cor}

%%%%%%%%%%%%%%%%%%%%%%%%%%%%%%%%%%%%%%%%%%%%%%%%%%%%%%%%%%%%%%%
\begin{example}\label{ex:clevershift}
 We continue Examples \eqref{ex:shift} and \eqref{ex:dumbshift}, illustrating Corollary \ref{prop:GenChen}.
 For  given $\lambda^{-}, \lambda^{+} $ the probability measure $\rho $ belongs to $\drecl $ if and only if 
\begin{equation*}
  \rho \left( f(\nmin+1,\npiu+1) \right) = \frac{1}{\lambda^{-} \lambda^{+}} \   \rho \left( f(\nmin,\npiu) \npiu \nmin \right) 
  \quad \forall f \in \BM^{\sharp}(\Z^2).
 \end{equation*}
 This improves Example \ref{ex:dumbshift}, where we obtained a redundant characterization.
\end{example}
%\begin{remark}
%The fact that the conditional distributions of a Poisson random vector can be characterized via 
%change of measure formula in the same spirit as Chen's characterization, appeals for a conditional version of Chen's method, 
%i.e. one might bound the distance between the conditional distributions of $\rho \in \PM(\Z^A) $ and the conditional distributions 
%of $ \pl $ by measuring to which extent the formulae \eqref{csetChen} fail to be true, in the same spirit as the Chen and Stein's method. 
%\end{remark}
%%%%%%%%%%%%%%%%%%%%%%%%%%%%%%%%%%%%%%%%%%%%%%%%%%%%%%%%%%%%%%%%%%%%%%%%%%%%%%%%%%%%%%%%%%%%%%%%%%%%%%%%%%%%%%%%%%%%%%%%%%%%%%%%%%%%%%%%%%%%%%%%%%%%%%%%%%%%%%%%%%%%%%%%%%%%%%%%%%%%%%
\section{Characterization of the reciprocal class} \label{sec:char}

\subsection{Main result}
We present here our main result: the reciprocal class $\Rec (\nu)$ associated to a compound Poisson process with jump measure $\nu$  is characterized as 
the set of all probabilities for which a family of transformations induces the same density, expressed in terms of the reciprocal invariants.
We have already introduced in the previous section the family of  reciprocal time-invariants. 
Let us now introduce the family of reciprocal space-invariants.
\begin{definition}
Let $\nu$ be a jump measure in  $\smch$ as defined in \eqref{ eq: charact}. 
For any $ \cv \in \C$ we call \textbf{reciprocal space-invariant} $\sinv$ the positive number 
\begin{equation*}
\sinv := \prod_{j=1}^{A} \left(\int_{0}^{1} \nu^{j}(t) dt \right)^{-c^j} .
\end{equation*}
 \end{definition}

\begin{remark}
In the time homogeneous case, $\nu^j_t \equiv \nu^j $,
$ \sinv = 1/  \prod_{j=1}^{A} ( {\nu^j} )^{c^j}$.
\end{remark}

We can now use these invariants to characterize the reciprocal class.

\begin{thm}\label{thm:char}
Let  $\nu \in \smch$ and $\QR\in\PM(\Omega) $. Then $\QR$ belongs to the reciprocal class $\Rec(\nu) $ if and only if
\begin{itemize}
\item[$i$)] For all $u \in \U$ and all $F \in \BM(\OO)$,
 \begin{equation}\label{eq:udns}
  \QR \Bigl( F \circ \pi_{u} \Bigr)  =\QR \Bigl( F \,  \GT \Bigr).
 \end{equation} 
\item[$ii$)] There exists a generating set $\CS \subseteq \C$ such that for every $\cv \in \CS$ and every $f \in \BM^{\sharp}(\Z^A)$, 
 the following identity holds:
\begin{equation}\label{eq:cdns}
\rho \Bigl(f \circ \shc  \Bigr)  = \sinv \,\  \rho  \Bigl( f \ \GCC \Bigr),
\end{equation}
where  $ \rho:= \QR \circ \NO^{-1 }\in \mathcal{P}(\N^A) $ is the law of $\NO$ under $\QR$.
\end{itemize}
\end{thm}
\vspace{3mm}

\begin{remark}\label{anytime}
Note that identities similar to \eqref{eq:cdns} hold for any $t \in ]0,1]$, i.e. any $\QR \in \Rec(\nu)$ satisfies (we assume a time homogeneous $\nu$, for
simplicity):
\begin{equation}\label{eq:ctdns}
\QR \left( f \circ \shc (\mathbf{N}_t) \right) = \sinv \ (1-t)^{-|\cv|}\ \QR \left(  (f \GCC ) (\mathbf{N}_t) \right), \quad \forall f \in
\BM^{\sharp}(\Z^A), 0<t \leq 1,
\end{equation}
where $|\cv| := \sum_{j =1}^Ac^j$.
However, the identities \eqref{eq:ctdns} do not contain enough information to characterize the reciprocal class as the time-invariants do not appear.
\end{remark}

\begin{proof}
($\Rightarrow$) Let $\QR \in \Rec(\nu)$ and $\PN^{\QR}$  be constructed as in Proposition \ref{prop:Q=hP}. Since there is no ambiguity, we
write $\PN$ rather than $\PN^{\QR}$. An application of the same proposition gives that $\QR << \PN$, and $h:= \frac{d\QR}{d\PN}$ is 
$\sigma(X_0,X_1)$-measurable.
Consider now  $ u \in \U $. By definition of $u$, for any $j$,
$N^j_{1} \circ \pi_u = N^j_{1}$, 
so that $(X_0,X_1) \circ \pi_u = (X_0,X_1), \, \PN \ps .$
 
We then consider $F \in \BM(\OO)$  and 
apply Proposition \ref{prop:udns} under the measure $\PN$, which leads  to 
\begin{eqnarray*}
\QR \Bigl( F \circ \pi_u \Bigr)&=& \PN \Bigl( (F \circ \pi_u) h(X_0,X_1)   \Bigr)  
= \PN \Bigl( (Fh(X_0,X_1)) \circ \pi_u \Bigr)\\ &=& \PN \Bigl( Fh(X_0,X_1) \GT \Bigr) \\ &=& \QR \Bigl(F \GT \Bigr).
\end{eqnarray*}
In a similar way, if $\cv \in \CS$, since $\CS \subseteq \C $ we have that $\mathbf{A}(\shc \NO) = \mathbf{A}\NO$. We observe that  $\PN ( \NO \in .| X_0=x) = \pl $, where
\begin{equation}\label{ident}
\lambda^j:= \int_{0}^{1} \nu^{j}(t) dt .
\end{equation}
For $f \in \BM^{\sharp}(\Z^A)$ and $\cv \in \CS $ we use Proposition \ref{prop:cycleshift}, observing that $\NO$ has law $\pl$ and is independent of $X_0$,
to obtain
\begin{eqnarray*}
\rho \Bigl(f \circ \shc  \Bigr)  &=& 
\QR \Bigl( f \circ   \shc (\NO )  \Bigr) \\
&=& \PN \Bigl( h(X_0,X_1) \ f\circ \shc \circ \NO  \Bigr) \\ &=& 
\PN \Bigl( h(X_0,X_0+\mathbf{A}( \shc \NO)) f \circ \shc \circ \NO  \Bigr) \\
&=& \PN\Bigl( \PN^{X_0}\Bigl(  h(X_0,X_0+\mathbf{A}(\shc \NO)) f \circ \shc \circ \NO \Bigr) \Bigr)    \\
 &=& \sinv \  \PN\Bigl( h(X_0,X_1) (f\GCC ) \circ \NO  \Bigr)= \ \sinv \ \rho  \Bigl( f \ \GCC \Bigr)
\end{eqnarray*}
and $ii$) is now proven.\\

($\Leftarrow$) 
We will show that $\QR$ satisfies $ii$) of Proposition \ref{prop:Q=hP}, which is equivalent to $\QR \in \Rec(\nu)$. We divide the proof in three steps. 
In a first step, we refer to the Appendix for the proof of the absolute continuity  of $\QR$ w.r.t. to
$\PN^{\QR}$, since it is quite technical. 
In a second step we prove that the density is $\sigma(X_0,\NO)$-measurable and in a third one we prove that this density is indeed 
$\sigma(X_0,X_1)$-measurable. 
For sake of clarity, since there is no ambiguity, we denote by $\PN$ the probability $\PN^{\QR}$.\\
%%%%%%%%%%%%%%%%%%%%%%%%%%%%%%%%%%%%%%%%%%%%%%%%%%%%%%%%%%%%%%%%%%%%%%%%%%%%%%%%%%%%%%%%%%%%%%%%%%%%%%%%%%%%%%%%%%%%%%%%%%%%%%%%%%%%%%%%%%%%%%%%%%%%%%%%%%%%%%%%
%%%%%%%%%%%%%%%%%%%%%
{\it Step 1: Absolute continuity}.\\
 See the Appendix. \\
%%%%%%%%%%%%%%%%%%%%%%%%%%%%%%%%%%%%%%%%%%%%%%%%%%%%%%%%%%%%%%%%%%%%%%%%%%%%%%%%%%%%%%%%%%%%%%%%%%%%%%%%%%%%%%%%%%%%%%%%%%%%%%%%%%%%%%%%%%%%%%%%%
{\it Step 2: The density $H:= \frac{ d \QR}{d \PN}$ is invariant under time change.}\\
We show that, for any $u \in \U$, $H$ is $\pi_u$-invariant, i.e. $H \circ \pi_u = H \ \PN \ps$. 
Since $\PN(\Omega)=1$ we have that $\pi_u$ is $\PN \ps$ invertible.  Applying the identity \eqref{eq:timepert}  under $\PN$ 
%and then point i) of the hypothesis 
we obtain, for any $F \in \BM(\Omega)$:
\begin{equation*}
\begin{split}
\PN \left(  F \,  H \circ \pi_u \right) &= \PN \left( (F \circ \pi_u^{-1} \, H)  \circ \pi_u \right)\\								
				&= \PN \Big( F \circ \pi_u^{-1} \, H \,  \GT \Big)\\
				&= \QR\Big(  F \circ \pi_u^{-1} \, \GT  \Big)\\
&= \QR \left( F \right) = \PN \left(  F\, H  \right)
\end{split}
\end{equation*}
which gives us the desired invariance, since $F$ is arbitrary.\\
We claim that this implies that $H$ is $\sigma(X_0,\NO)$-measurable, i.e. that there exists a function $h: \R \times \N^A \longrightarrow \R^+$  such that 
\begin{equation*}
 H = \frac{ d \QR}{d \PN} = \frac{d\QR \circ (X_0,\NO)^{-1}}{d\PN  \circ (X_0,\NO)^{-1}} = h(X_0,\NO)  \quad \PN \ps .
\end{equation*}
 This is true since, given any two $ \om, \om' \in \Omega$  with the same initial state and the same number of jumps of each type, 
 one can construct $u \in \U$ such that $\om' = \pi_u (\om)$. \\
%%%%%%%%%%%%%%%%%%%%%%%%%%%%%%%%%%%%%%%%%%%%%%%%%%%%%%%%%%%%%%%%%%%%%%%%%%%%%%%%%%%%%%%%%%%%%%%%%%%%%%%%%%%%%%%%%%%%%%%%%%%%%%%%%%%%%%%%%%%%%%%%%
{\it{Step 3}: The density $H$ is invariant under shifts in $\CS$.} \\
Let us recall that $\PN (\NO \in . | X_0 = x) = \pl$, where $\lambda$ is given by \eqref{ident}. Under our assumption we might apply Corollary
\ref{prop:GenChen} 
to $\pl = \PN ( \NO \in. |X_0=x)$ and $\rho^x = \QR (\NO \in . |X_0=x )$. 
We obtain that the conditional density $ \frac{d \rho^x}{d\pl} $ is  $\mathbf{A}\NO $-measurable $\QR_{0} \ \ps$ and, 
by mixing over the initial condition, that
 $\frac{d\QR \circ (X_0,\NO)^{-1}}{d\PN  \circ (X_0,\NO)^{-1}} = \frac{d\QR}{d\PN}$ is $\sigma(X_0,\mathbf{A}\NO) = \sigma(X_0,X_1)$-measurable.
\end{proof}

\subsection{Comparing reciprocal classes through invariants}

In what follows and in the next subsections, we consider jump measures $ \nu \in \smch$ which are time-homogeneous. In that case  we identify $\nu$ with the
vector  \\ $(\nu^1(0),...,\nu^{A}(0)) \in \R^{A}_{+}$. \\
We present  in Proposition \ref{geometricchar} a set of explicit necessary and sufficient conditions for two compound Poisson processes $\PNX$ and
$\mathbb{P}^x_{\tilde{\nu}}$ to have the same bridges, or equivalently, to belong to the same reciprocal class. 
A more implicit condition has been presented by R. Murr in his PhD thesis, see \cite{Murr12}, Proposition 7.50. 
Our result is in the spirit of \cite{LMR}, where two counting processes where shown 
to have the same bridges if and only if their reciprocal (time-)invariants coincide. \\
%%%%%%%%%%%%%%%%%%%%%%%%%%%%%%%%%%%%%%%%%%%%%%%%%%%%%%%%%%%%%%%%%%%%%%%%%%%%%%%%%%%%%%%%%%%%%%%%%%%%%%%%%%%%%%%%%%%%%%%%%%%%%

 %Note that there is no ambiguity since $\nu^j(t)$ does not depend on $t$. 
In the next proposition we denote by $\C^{\bot}$ the orthogonal complement of the affine hull of $\C$, and the logarithm  of the vector $\nu \in \R^{A}_{+}$, denoted by $\log(\nu)$,
has to be understood componentwise. 

\begin{prop}\label{geometricchar}
%%%%%%%%%%%%%%%%%%%%%%%%%%%%%%%%%%%%%%%%%%%%%%%%%%%%%%%%%%%%%%%%%%%%%%%%%%%%%%%%%%%%%%%%%%%%%%%%%%%%%%%%%%%%%%%%%%%%%%%%%%%%%%%%%%%%%%%%%%%%%%%%%%%%%%%%%%%%%%%%%%%%%%%%%%%%%%%%%%%%%%%%%%%%%%%%%%%%%%%%%%%%%%
Let $x \in \RD$, $\nu, \tilde{\nu} \in \R^{A}_{+} $ and $\CST$ be a lattice basis of $\C$.
The following assertions are equivalent:
\begin{itemize}
\item[$i$)] $ \mathbb{P}^{x}_{\tilde{\nu}} \in \Rec(\nu)$.
\item[$ii$)] For every $\cv \in \C^{*}$ the equality $\Phi_{\nu}^{\cv} = \Phi_{\tilde{\nu}}^{\cv}$ holds.
\item[$iii$)] There exists $v \in \C^{\bot} $ such that $ \log (\tilde{\nu}) = \log(\nu) + v $.
\end{itemize}
\end{prop}
\begin{proof}
 $i) \Rightarrow ii) $ By applying  \eqref{eq:cdns} and  the trivial fact that $\PR_{\tilde{\nu}}^x \in \Rec(\tilde{\nu})$, we have
\begin{equation}
 \Phi_{\tilde{\nu}}^{\cv} \, \PR_{\tilde{\nu}}^{x} (f) = \PR_{\tilde{\nu}}^{x} (f \circ \shc \circ \NO ) = \Phi_{\nu}^{\cv} \, \PR_{\tilde{\nu}}^{x} (f ) , \quad \forall f \in
\BM^{\sharp}(\Z^A),
\end{equation} 
and $ii)$ follows.\\
 $ii) \Rightarrow i)$
Observe that since $\C^{*}$ is a lattice basis, any $\cv \in \C$ can be written as an integer combination of the elements of $\CST$, i.e. 
$\cv= \sum_{\cv^{*} \in \CST} z_{\cv^{*}} \cv^{*}$, $z_{\cv^*} \in \Z$. Therefore all the reciprocal space-invariants coincide since
 \begin{equation}
 \Phi_{\nu}^{\cv} = \prod_{\cv^* \in \CST} (\Phi^{\cv^*}_{\nu})^{z_{\cv^*}} = \prod_{\cv^* \in \CST} (\Phi^{\cv^*}_{\tilde{\nu}})^{z_{\cv^*}} = \Phi_{\tilde{\nu}}^{\cv},
\quad \forall \cv \in \C.
 \end{equation}
 With a similar argument as above one proves that the identity \eqref{eq:cdns} is satisfied under $\PR_{\tilde{\nu}}^x$. 
 The functional equation \eqref{eq:udns} is trivially satisfied by $\PR_{\tilde{\nu}}^x$ because $  \Xi^{\nu} \equiv \Xi^{\tilde{\nu}} =1$. 
 The conclusion follows by applying  Theorem \ref{thm:char}.\\
$ii) \Leftrightarrow iii) $ We just observe that the equality $\sinv = \Phi^{c}_{\tilde{\nu}}$ is equivalent to
\begin{equation*}
 \sum_{j=1}^{A}\log(\nu^j) c^j  = \sum_{j=1}^{A}\log(\tilde{\nu}^j) c^j .
\end{equation*} 
Since a lattice basis $\CST$ of $\C$ is a linear basis of the affine hull of $\C$ $ii)$   is equivalent to the fact that $\log(\nu)$ and $\log(\tilde{\nu})$ 
have the same projection onto $\C$,
which is equivalent to $iii)$.
 \end{proof}
\begin{example}
Continuing on Example \ref{ex:clevershift}, two time-homogeneous compound Poisson processes  with jumps in $\A=\{-1,1\}$ and rate $\nu= (\nu^- ,\nu^+)$
resp. $\tilde{\nu}=( \tilde{\nu}^-,\tilde{\nu}^+)$ have the same bridges if and only if
\begin{equation*}
\nu^- \nu^+ = \tilde{\nu}^-\tilde{\nu}^+ .
\end{equation*}
\end{example}
\begin{example}
Let $\A = \{ -1,3 \}$ and define two time-homogeneous compound Poisson processes  with jumps in $\A$ and rate $\nu= (\nu^- ,\nu^+)$
resp. $\tilde{\nu}=( \tilde{\nu}^-,\tilde{\nu}^+).$
They have the same bridges if and only if
\begin{equation*}
(\nu^-)^3 \nu^+= (\tilde{\nu}^-)^3\tilde{\nu}^+ .
\end{equation*}
\end{example}

\begin{example}\label{hex}
Let $\A=\{a^1,..., a^6 \}$ be the vertices of an hexagon, see the Figure 2: 
\begin{equation}
a^i = \big(\cos(\frac{2\pi}{6} (i-1) ) , \sin(\frac{2 \pi}{6} (i-1))\big) \in \R^2,\quad  i=1,...,6.
\end{equation}
Then a basis of $\C$ is:
\begin{equation}
\CST = \{ \mathbf{e}_1 + \mathbf{e}_4, \mathbf{e}_2 + \mathbf{e}_5 , \mathbf{e}_1 + \mathbf{e}_3 + \mathbf{e}_5, \mathbf{e}_2 + \mathbf{e}_4 + \mathbf{e}_6  \}.
\end{equation}
By Proposition \ref{geometricchar},  $\PNX$ with jump rates $\nu=(\nu^1,...,\nu^6)$ belongs to $ \mathfrak{R}(\tilde{\nu})$ if and only if 
\begin{equation*}
\begin{cases}
\nu^1 \nu^4 =  \tilde{\nu}^1\tilde{\nu}^4,\\
  \nu^2 \nu^5 =  \tilde{\nu}^2 \tilde{\nu}^5,\\
   \nu^1 \nu^3 \nu^5 = \tilde{\nu}^1 \tilde{\nu}^3 \tilde{\nu}^5,\\
    \nu^2 \nu^4 \nu^6 = \tilde{\nu}^2 \tilde{\nu}^4 \tilde{\nu}^6
    \end{cases}
\end{equation*}

\end{example}

\begin{figure}[h!] \label{fig:hexagon}
\centering
\includegraphics[height=3.5cm, width=\linewidth]{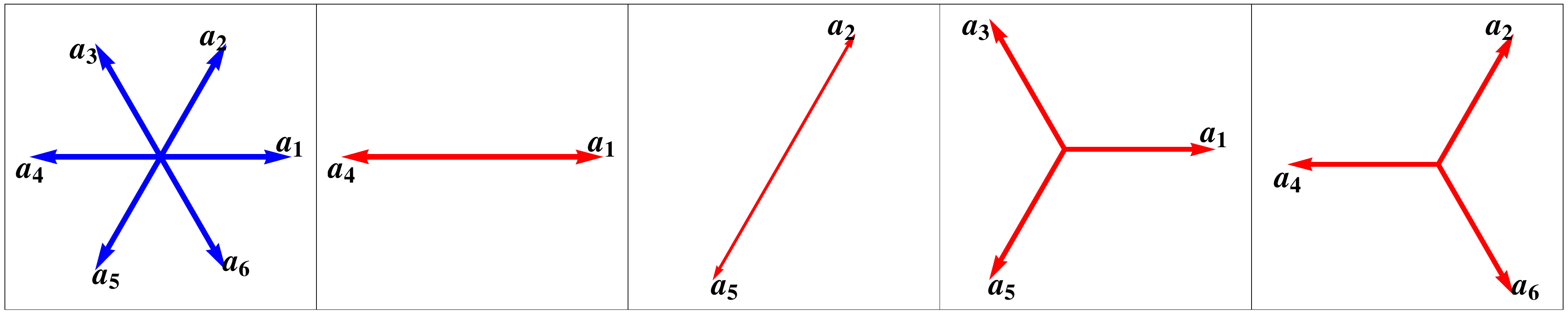}
\caption{A representation of the vectors of $\A$ and of the incidence vectors of  $\CST$}
\end{figure}

\vspace{4mm}
%%%%%%%%%%%%%%%%%%%%%%%%%%%%%%%%%%%%%%%%%%%%%%%%%%%%%%%%%%%%%%%%%%
%%%%%%%%%%%%%%%%%%%%%%%%%%%%%%%%%%%%%%%%%%%%%%%%%%%%%%%%%%%%%%%%%
\subsection{An interpretation of the reciprocal space-invariants}
%%%%%%%%%%%%%%%%%%%%%%%%%%%%%%%%%%%%%%%%%%%%%%%%%%%%%%%%%%%%%%%%

We aim at an interpretation of the space-invariants for  a time-homogeneous jump measure $ \nu  \in \smch $ under the geometrical assumption 
$ii$) of Proposition \ref{posray}:
\begin{equation}\label{geoass}
\C \ admits \  a \  lattice \  basis \  \CST \ included \ in \  \N^A  .
\end{equation}
A lattice basis satisfying \eqref{geoass} is a generating set for $\C$. 
Therefore it is sufficient to interpret the invariants  $\Phi^{\cv}_{\nu}$ for $\cv \in \CST$.\\
Assumption \eqref{geoass} is not only  natural in view  of the interpretation we will give in  Proposition \ref{loopint} but it 
 is satisfied in many interesting situations. One can prove that this is the case when $\A \subseteq \Z$ and $\A$ contains at least 
 one negative and one positive jump. 
Assumption \eqref{geoass} also holds in
several situations when $d>1$, e.g. in the setting of Example \ref{hex}.\\
In the context of diffusions, various physical interpretation of the reciprocal invariants have been given, 
mainly based on analogies with Stochastic Mechanics, see \cite{CruZa91}, \cite{LeKre93}, \cite{Th93} and \cite{TZ97}. 
Regarding jump processes, the only
interpretation known to the authors was given by R. Murr \cite{Murr12}. 
Inspired by  \cite{PrZamb04} he related  the reciprocal time-invariant associated to a counting process 
(the space-invariants trivialize) with a stochastic
control problem, whose cost function is expressed in terms of the invariant.
% where a general characterization of Bernstein processes based on stochastic control for jump processes is presented. \\
% The interpretation we propose here does not rely on any physical consideration, and it has the advantage that it provides an asymptotic expansion which remains valid for any probability in the reciprocal class.

We propose here a different interpretation of the invariants as infinitesimal characteristics, 
based on  the short-time expansions for the probability that the process makes a cycle around its current state. We believe this interpretation to be
quite general, and we are currently working on various extensions.

\vspace{0.1cm}

To be precise, let us  define the concept of cycle we use here. In the rest of this section, a basis $\CST $ satisfying \eqref{geoass} is fixed.%For a given element $ \cv \in \C$ we define its length $l(c):=\sum_{j=1}^A |c^j|$.
\begin{definition}
A cycle is a finite sequence $\cy:=( x_k )_{0 \leq k \leq l}$ such that
\begin{itemize}
\item[i)] \  $x_k - x_{k-1} \in \A , \ \ \  1 \leq k \leq l$,
\item[ii)]  $ x_l=x_0=0 $.
\end{itemize}
\end{definition}

To each cycle $\cy$ we can associate an element $\mathbf{N}(\cy) \in \C \cap \N^A$ by counting how many times each jump occurred in the cycle, thus neglecting the order at which they occurs:
\begin{equation*}\label{unord}
\mathbf{N}(\cy)^j := \sharp  \{ k: x_k - x_{k-1}  = a^j \}, \quad  1\leq j \leq A,
\end{equation*}  
where $\sharp E$ denotes the number of elements of a finite set $E$.
Note that, for a given $\cv \in \C $, we can construct a cycle $\cy$ such that $\mathbf{N}(\cy) = \cv$ if and only if $\cv \in \N^A$. 
Therefore, under assumption \eqref{geoass},   $\mathbf{N}^{-1}(\cv)$ is non empty  for any $\cv \in \CST$.  

\begin{figure}[h!]
\centering
\includegraphics[height=4.2cm, width=\linewidth]{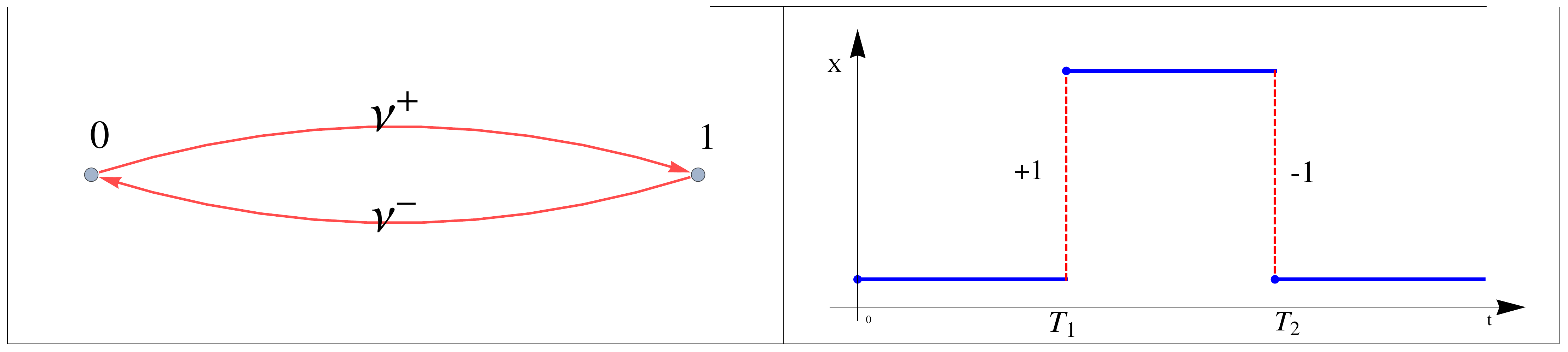}
\caption{Here $\A = \{ -1,1 \}$ and $\C= \left( 1,1 \right) \Z $. 
Left:  A representation of the cycle $\gamma = \{ 0,1,0\}$ satisfying $\mathbf{N}(\gamma)=\left( 1,1 \right)$.
 Right: A typical path in $L_{\varepsilon}^{\gamma}$. 
The probability of $L_{\varepsilon}^{\gamma}$ is equivalent  to $(\nu^+ \nu^-)\varepsilon^2 $ over the whole reciprocal class, as $\varepsilon \rightarrow 0$. 
}
\end{figure}

\vspace{0.1cm}
\begin{definition}
We define the trace $\gamma_{\varepsilon}(\om)$ of a path $\om \in \OO$ as the ordered sequence formed by the displacements
from the initial position up to time $\varepsilon$:
\begin{equation*}
 \trc (\om) = (0,X_{T_1}-X_0,...,X_{T_{|N|_{\varepsilon}}}-X_0) .
\end{equation*}
The subset of paths whose trace coincides with a given cycle $\cy$ over a small time interval $[0,\varepsilon]$ is denoted by
\begin{equation*}
L_{\varepsilon}^{\cy} :=  \{ \om: \trc (\om) = \cy \} .
\end{equation*}
\end{definition}

Finally, we introduce the usual time-shift operator on the canonical space:
\begin{equation*}
\tau_{t}: \mathbb{D}([0,1],\RD) \longrightarrow  \mathbb{D}([0,1-t],\RD), \quad \tau_{t}(\om)_s = \om_{t+s}, \forall \  0 \leq s \leq 1-t .
\end{equation*}

The following short-time expansion holds under the compound Poisson process. 
\begin{prop}\label{fwdloop}
Let $\nu \in \smch$ be a time-homogeneous jump measure, $x, x_0 \in \RD$. 
Then for any time $t\geq 0$, $\cv \in \CST$ and  any cycle $\cy$ with $ \mathbf{N} (\cy)=\cv$,  we have:
\begin{equation*}\label{refshtime}
\PN^{x_0}(\tau_{t}(X) \in L_{\varepsilon}^{\cy}| X_t = x) = \frac{1}{\sinv \  |\cv|!}  \varepsilon^{|\cv|} + o(\varepsilon^{|\cv|}) \quad  as \  \varepsilon
\rightarrow 0 
\end{equation*}
where $|\cv| = \sum_{j=1}^{A} c^j$.
\end{prop}
\begin{proof}
First observe that w.l.o.g. $t=0$, the general result following from  the Markov property of $\PN^{x_0}$. 
For simplicity, we denote by $\bar \nu$ the total jump rate $\sum_{j=1}^{A} \nu^j$. Moreover, we denote by $j(k)$ the unique element of $\{ 1,...,A\}$ such that $X_{T_k}-X_{T_{k-1}} = a^{j(k)}$. With an elementary computation based on the explicit distribution of $\PN^{x_0}$:
\begin{eqnarray*}
\PN^{x}( L_{\varepsilon}^{\cy} ) 
&= &\PN^{x}\Big(\{ |N|_{\varepsilon} = |\cv| \}  \cap \bigcap_{k=1}^{|\cv|}\{ X_{T_k}-X_{T_{k-1}} = a^{j(k)} \} \Big) \\
&=& \exp(- \varepsilon \bar \nu ) \frac{( \varepsilon \bar \nu)^{|\cv|}}{|\cv|!} \prod_{k=1}^{|\cv|} \frac{\nu^{j(k)}}{\bar \nu} 
= \exp(- \varepsilon \bar \nu) \varepsilon^{|\cv|} \prod_{j=1}^{A} (\nu^j)^{\sharp \{ k :  j(k) = j \}}
\\ &=& \exp(- \varepsilon \bar \nu) \frac{\varepsilon^{|\cv|}}{|\cv|!} \prod_{j=1}^{A} (\nu^j)^{c^j}
= \exp(- \varepsilon \bar \nu ) \frac{1}{\sinv \,  |\cv|!}  \varepsilon^{|\cv|}
\end{eqnarray*}
from which the conclusion follows.
\end{proof}
%%%%%%%%%%%%%%%%%%%%%%%%%%%%%%%%%%%%%%%%%%%%%%%%%%%%
Even more interesting, the same time-asymptotics holds under any $\QR \in \Rec(\nu)$ and in particular under any bridge $\PNB$.
%%%%%%%%%%%%%%%%%%%%%%%%%%%%%%%%%%%%%%%%%%%%%%%%%%%%%%%%%%%%%%%%%%%%%%%%%%%%%%%%%%%%%%%%%%%%%%%%%%%%%%%%%%%%%%%%%%%%%%%
\begin{prop}\label{loopint}
Let $\nu \in \smch$ be a time-homogeneous jump measure and $\QR \in \Rec(\nu)$. 
Then for any time $t\geq 0$, $\cv \in \CST$ and  any  cycle $\cy$ with $ \mathbf{N} (\cy)=\cv$,  we have:
\begin{equation*}
\QR \ps \quad  \QR \Bigl( \tau_t(X) \in L_{\varepsilon}^{\cy} \Big| X_{t} \Bigr)  =  \frac{1}{\sinv \ |\cv|!}  \varepsilon^{|\cv|} + o(\varepsilon^{|\cv|}) \quad  as \  \varepsilon \rightarrow 0 
\end{equation*}
\end{prop}
%%%%%%%%%%%%%%%%%%%%%%%%%%%%%%%%%%%%%%%%%%%%%%%%

\begin{proof}

Assume that $\QR \in \Rec(\nu)$. Observe that w.l.o.g we can assume that $\QR_{0} = \delta_{x_0}$ for some $x_0 \in \RD$, 
the general result following by mixing over the initial condition.  
Then by Proposition \ref{prop:Q=hP},  $d\QR /d\PN^{x_0} = h(X_1)$. We first show the identity:
\begin{equation}\label{deccondexp}
\PN^{x_0}\Big( \mathds{1}_{\{\tau_t(X) \in L_{\varepsilon}^{\cy} \} } h(X_1) \Big| X_t \Big) = 
\QR \Big( \mathds{1}_{ \{ \tau_t(X) \in L_{\varepsilon}^{\cy} \} } |X_{t} \Big)  \PN^{X_t}\Big( h(X_{1-t}) \Big) .
\end{equation}
Indeed, let us take any test function of the form $\mathds{1}_{\{X_t \in A\}}$. We have:
\begin{equation*}
\begin{split}
\PN^{x_0}\Big( \mathds{1}_{ \{ \tau_t(X) \in L_{\varepsilon}^{\cy} \} }  \ h(X_1) &  \mathds{1}_{\{X_t \in A\}}  \Big)  =  \QR\Big( \mathds{1}_{ \{ \tau_t(X) \in L_{\varepsilon}^{\cy} \} } \  \mathds{1}_{\{ X_t \in A\} } \Big) \\
\vspace{0.2cm}
&=  \QR( \  \QR(\mathds{1}_{ \{ \tau_t(X) \in L_{\varepsilon}^{\cy} \}} |X_t ) \    \mathds{1}_{\{ X_t \in A\} } ) \\
\vspace{0.2cm}
 &= \PN^{x_0} \Big( \  \QR(\mathds{1}_{ \{ \tau_t(X) \in L_{\varepsilon}^{\cy} \}} |X_t ) \    h(X_1) \mathds{1}_{\{ X_t \in A\} } \Big)  \\
\vspace{0.2cm} 
 &= \PN^{x_0} \Big( \ \QR( \mathds{1}_{ \{ \tau_t(X) \in L_{\varepsilon}^{\cy} \}} |X_t )  \ \PN^{x_0}(h(X_1) |X_t )   \    \mathds{1}_{\{ X_t \in A\} } \Big)
 \end{split}
\end{equation*}
from which \eqref{deccondexp}  follows. Consider now the left hand side of \eqref{deccondexp}. 
We have, by applying the Markov property and the fact that $\gamma$ is a cycle:
\begin{eqnarray*}
\PN^{x_0} \Big(  h(X_1)  \ \mathds{1}_{\{\tau_t(X) \in L_{\varepsilon}^{\cy} \} }  |X_t \Big)
 &=& \PN^{x_0} \Big(   \PN^{x_0}( h(X_1) | \mathcal{F}_{[t,t+\varepsilon]} )  \mathds{1}_{\{ \tau_t(X) \in L_{\varepsilon}^{\cy} \} }    | X_t \Big) \\
\vspace{0.2cm}
 &=& \PN^{x_0} \Big(  \PN^{X_{t+\varepsilon}} (h(X_{1-(t+\varepsilon) })    )  \ \mathds{1}_{\{\tau_t(X) \in L_{\varepsilon}^{\cy} \}  }   |X_t \Big) \\
\vspace{0.2cm} 
  &=& \PN^{x_0}\Big(  \PN^{X_{t}} ( \ h(X_{1-(t+\varepsilon) })  )  \ \mathds{1}_{\{\tau_t(X) \in L_{\varepsilon}^{\cy} \} }  \ |X_t \Big) \\
  \vspace{0.2cm}
   & =& \PN^{x_0}\Big(   \mathds{1}_{\{\tau_t(X) \in L_{\varepsilon}^{\cy} \} } |X_t \Big) \  \PN^{X_{t}} \Big(h(X_{1-(t+\varepsilon) }) \Big).
\end{eqnarray*}
Applying \eqref{deccondexp} and Proposition \ref{fwdloop}  and the continuity  of $$( \om,t,.) \mapsto \PN^{X_t}( h(X_{1-.})) $$ we obtain: 
\begin{equation}\label{implicitfinal}
 \frac{1}{\sinv \ |\cv|!}  \ \PN^{X_{t}}(h(X_{1-t })  ) = \lim_{\varepsilon \rightarrow 0 } \varepsilon^{-|\cv|} \, \QR ( \mathds{1}_{ \{ \tau_t \in L_{\varepsilon}^{\cy} \} } |X_{t} )   \ \PN^{X_t}( h(X_{1-t}) )
\end{equation}
We observe  that $ \PN^{X_t}( h(X_{1-t}) ) = d{\QR}_t/ d(\PN^{x_0})_{t}  $ and therefore it is strictly positive $\QR \ps$ .
This allows us to divide on both sides by $\PN^{X_t}( h(X_{1-t}) )$ and the conclusion follows.
\end{proof}

We have thus shown that each element of the reciprocal class has the same probability to {\em spin around} its current state in a very short time interval.
\begin{remark}
 In the statement of Proposition \ref{loopint} we could have replaced $X_t$ with $\mathcal{F}_t$, i.e. the following asymptotics holds true:
 \begin{equation*}
 \QR( \tau_{t}(X) \in L_{\varepsilon}^{\cy}  | \mathcal{F}_{t} ) =  
  \frac{1}{\sinv \ |\cv|!} \varepsilon^{|\cv|} + o(\varepsilon^{|\cv|}) \quad  as \  \varepsilon \rightarrow 0 .
 \end{equation*}
 \end{remark}
\vspace{0.7cm}
%%%%%%%%%%%%%%%%%%%%%%%%%%%%%%%%%%%%%%%%%%%%%%%%%%%%%%%%%%%%%%%%%%%%%%%%%%%%%%%%%%5
\section{Appendix}
\begin{proof} $\textit{(Step 1 in Theorem \ref{thm:char})}$\\
 We first observe that it is sufficient to prove that  
$$
\QR(.|\NO=\nv) << \PN(.|\NO =\nv) \, \textrm{ for all } \nv \textrm{ such that }\QR(\NO =\nv)>0.
$$
 To this aim, we use an approximation argument.

Let us fix $\nv$ and construct a discrete (dyadic) approximation of the jump times. 
For $m \geq  \max_{j=1,...,A} \log_{2} (n^j)+1:=\bar{m}$ , $\mathcal{D}^m$ is composed by $A$ ordered sequences of dyadic numbers, the j-th sequence having length $n^j$:
%%%%%%%%%%%%%%%%%%%%%%%%%%%%%%%%%%%%%%%%%%%%%%%%%%%%%%%%%%%%%%%%%%%%%%%%%%%%%%%%%%%%%%%%%%%%%%%%%%%%%%%%%%%%%%%%%%%%%%%%%%%%%%%%%%%%%%%%%%%%%%%%%%%%%%
\begin{equation*}
\mathcal{D}^{m}:= \left\{ \K = (k^j_i)_{j \leq A,  i \leq n^j} : \,   k^j_i \in 2^{-m} \N  , 0 < k^{j} _{i-1} < k^{j}_{i} \leq 1, \, \forall j \leq A, \forall i \leq  n^j  \right\}
\end{equation*} %%%%%%%%%%%%%%%%%%%%%%%%%%%%%%%%%%%%%%%%%%%%%%%%%%%%%%%%%%%%%%%%%%%%%%%%%%%%%%%%%%%%%%%%%%%%%%%%%%%%%%%%%%%%%%%%%%%%%%%%%%%%%%%%%%%%%%%%%%%%%%%%%%%%%%%%%%%%%%%%%%%%%%
For $\K \in \mathcal{D}^m$ we define the subset of trajectories whose jump times are localized around $\K$:
\begin{equation}
O^{m}_{\K} = \left\{  \NO = \nv \right\} \cap \bigcap_{\stackrel{j \leq A }{ i \leq n^j}} \left\{ 0 \leq k^j_i - T^j_i  < 2^{-m} \right\}
\end{equation}
%%%%%%%%%%%%%%%%%%%%%%%%%%%%%%%%%%%%%%%%%%%%%%%%%%%%%%%%%%%%%%%%%%%%%%%%%%%%%%%%%%%%%%%%%%%%%%%%%%%%%%%%%%%%%%%%%%%%%%%%%%%%%%%%%%%%%%%%%%%%%%%%%%%%%%%%%%%%%%%%%%%%%%%%%%%%%%%%%%%%%
Moreover, as a final preparatory step, we observe for every $ m \geq \bar{m} $, $\K,\K' \in \mathcal{D}^{m}$ 
one can easily construct $ u \in \U $ such that:
\begin{equation}\label{umk}
u(j,t) =  t+k'^j_i-k^j_i , \quad \forall j \leq A, i \leq n^j \ and \  t \ s.t. \ 0 \leq  k^j_i - t < 2^{-m}
\end{equation}
We can observe that \eqref{umk} ensures $\dot{u}(j,T^j_i) =1 $ over $O^{m}_{\K}$, and that $O_{\K'}^m = \pi_u^{-1}(O^{m}_{\K})$.
We choose $F= \mathds{1}_{ O^{\nv}_{\K'} } \mathds{1}_{\{\NO=\nv \}}/\QR(\NO = \nv)$ and $u$ as in \eqref{umk} and apply \eqref{eq:udns} to obtain
:
\begin{equation*}
\begin{split}
\QR \Big(  O^{m}_{\K'}  | & \NO = \nv \Big) = \QR \Big(  \left\{ \om : \pi_{u}(\om) \in O^{m}_{\K}\right\} | \NO=\nv  \Big) \\ 
&= \QR \Big( \mathds{1}_{O^{m}_{\K}} \GT   \Big| \NO = \nv \Big) \\
&\geq   C \ \QR \Big( O^{m}_{\K}| \NO = \nv \Big),
\end{split}
\end{equation*}
where
\begin{equation}
C:= \Big( \inf_{s,t \in [0,1], j \leq A } \tinv(j,s,t) \Big)^{\sum_{j} \nv_j} >0
\end{equation}
since $\nu \in \smch$.
With a simple covering argument we obtain, for all $ m \geq \bar{m}$ and  $ \K \in \mathcal{D}^m $,
\begin{multline*}
\sharp \mathcal{D}^m \min \{ 1,\frac{1}{C} \} \QR \left(  O^{m}_{\K}  | \NO=\nv\right) \\  \leq  \QR \left( O^{m}_{\K} | \NO=\nv \right) + \sum_{\stackrel{\K' \in \mathcal{D}^m}{\K' \neq  \K}}  \QR \left(  O^{m}_{\K'} |\mathbf{\NO} = \nv \right)  \leq 1 .
\end{multline*}
It can be shown with a direct computation that $\frac{1}{|\mathcal{D}^m|} \leq C' \PN ( O^m_{\K} | \mathbf{\NO}=\nv)$ for some $C'>0$ uniformly in $m, \K \in \mathcal{D}^m$ (the proof is given in Lemma \ref{lem:count}). Therefore there exists a constant $C^{''}>0$ such that:
\begin{equation*}\label{abscont}
\QR(O^{m}_{\K}|\NO=\nv) \leq C^{''} \ \PN (O^{m}_{\K}|\NO=\nv),\quad \forall  m \geq \bar{m} , \K \in \mathcal{D}^{m} .
\end{equation*}
With a standard approximation argument, using the fact that $\QR(\Omega) =1$, we can extend the last bound to any measurable set. This completes the proof of the absolute continuity.  \\
\end{proof}
%%%%%%%%%%%%%%%%%%%%%%%%%%%%%%%%%%%%%%%%%%%%%%%%%%%%%%%%%%%%%%%%%%%%%%%%%%%%%%%%%%%%%%%%%%%%%%%%%%%%%%%%%%%%%%%%%%%%%%%%%%%%%%%%%%%%%%%%%%%%%%%%%%%%%%%%%%%%%%%%%%%%%%%%%%%%%%%%%%%%%%%%%%%%%%%%%%%%%%
\begin{lemma}\label{lem:count}
Let $\mathcal{D}^m$ and $\PN$ as before. Then there exists a constant $C^{'}$ such that for $m$ large enough, 
\begin{equation*}
C^{'} \ \PN(O^{m}_{\mathbf{k}}|\NO= \nv ) \geq \frac{1}{\sharp \mathcal{D}^{m}}
\end{equation*}
\end{lemma}
\begin{proof} 
We want to prove that,for $\nv \in \N^A $ :\begin{equation}\label{comb}
 \frac{1}{\sharp \mathcal{D}^m} \leq C' \PN(O^{m}_{\K}| \NO =\nv), \quad \forall \ m \geq \max_{j\leq A} \log(n^j)+1 , \, \K \in \mathcal{D}^m
\end{equation}
We can first compute explicitly $\sharp \mathcal{D}^m$ with a simple combinatorial argument: 
each $\K \in \mathcal{D}^m$ is constructed by choosing  $n^j$ dyadic intervals, $ j \leq A$, and ordering them. Therefore
\begin{equation}
\sharp \mathcal{D}^m = \prod_{j=1}^{A} \binom{2^{m}}{n^j}.
\end{equation}
On the other hand, we observe that defining $\tilde{\nu}(dxdt) =\sum_{j=1}^{A}\delta_{a^j}(dx) \otimes dt$, 
 $\PR_{\tilde{\nu}}$ is equivalent to $\PN $,  and therefore, we can prove \eqref{comb} replacing $\PN$ with $\PR_{\tilde{\nu}}$.
To do this, for each $\K \in \mathcal{D}^m$ we define the function:
\begin{eqnarray*}
\delta :  \{ 1,..., 2^{m}  \} \times  \{1,...,A \}  \longrightarrow \{ 0,1 \} \\
 \delta ( i,j) := \begin{cases} 1, & \quad \mbox{if $ i \in \{ 2^{m}\K^j_1,...,2^{m} \K^j_{n^j} \} $} \\ 0 , & \quad \mbox{otherwise .}\end{cases}
\end{eqnarray*}
 Then, using the explicit distribution of $\PR_{\tilde{\nu}}$,
\begin{eqnarray*}
 &&\PR_{\tilde{\nu}}(O^m_{\K} | \NO = \nv ) \\ 
 &&= \PR_{\tilde{\nu}} \Big( \bigcap_{(i,j) \in \{1,..,2^m \} \times  \{1,..,A\}} \{N^j_{\frac{i}{2^m}} -
N^j_{\frac{i}{2^m}} = \delta(i,j) \} | \NO =\nv \Big)  \\
&& = \exp(A)\exp(-2^{-m})^{2^{m}A} (2^{-m})^{(\sum_j n^j )}\prod_{j=1}^A n^j! =  \prod_{j=1}^A  2^{-mn^j}  n^j!  
\end{eqnarray*}
It is now easy to see that there exists a constant $C_0>0$ such that:
\begin{equation*}
 \binom{2^m}{n^j} \geq C_0 \frac{2^{m n^j}}{ n^j!} ,  \quad \forall \ j \leq A , \  m \geq \max_{j=1,...,A} \log(n^j)+1, \ \K \in \mathcal{D}^m 
\end{equation*}
from which the conclusion follows.
\end{proof}

%%%%%%%%%%%%%%%%%%%%%%%%%%%%%%%%%%%%%%%%%%%%%%%%%%%%%%%%%%%%%%%%%%%%%%%%%%%%%%%%%%%%%%%%%%%%%%%%%%%%%%%%%%%%%%%%%%%%%%%%%%%%%%%%%%%%%%%%%%%%%%%%%%%%%%%%%%%%%%%%%%%%%%%%%%%%%%%%%%%%%%%%%%%%%%%%%%%%%%%%%%%%%%%

\begin{proof} (\textit{ of Proposition} $\ref{posray}$)

\begin{itemize}
\item[i)]
 Let $\nv \in \N^A, \mathbf{m} \in \LF$. Since $\CST$ is a lattice basis there exists $\cv_1,...,\cv_K \subseteq (\CST \cup -\CST )^K $ such that, if we define recursively 
\begin{equation*}
 w_0 = \nv, \quad w_k = \theta_{\cv_k} w_{k-1}
\end{equation*}
then we have that $w_K = \mv $.  Let us consider $l$ large enough such that 
\begin{equation}\label{trans}
l  \ \min_{j=1,...,A} \bar{c}^j \geq |\min_{\stackrel{j=1...,A}{k=1,...,K}}w^j_k| .
\end{equation}
We then consider the sequence $w'_k$, $k = 0,...,K+2l$ defined as follows:
\begin{equation*}
 w'_k= \begin{cases} \theta_{\bar{\cv}}w'_{k-1}, \quad & \mbox{if $1 \leq k \leq l$}\\ \theta_{\cv_{k-l}} w'_{k-1}, \quad & \mbox{if $l+1 \leq k \leq K+l$}  \\ \theta_{-\bar{\cv}} w'_{k-1} \quad & \mbox{if $K+l+1 \leq k \leq K+2l$} .\end{cases}  
\end{equation*}
It is now easy to check, thanks to condition \eqref{trans}  that  
\begin{equation*}
 \ w'_k \in \LF \quad \forall \  k \leq K+2l  .
 \end{equation*}
 
Since all the shifts involved in the definition of $w'_k$ are associated to vectors in $\CST \cup - \CST$ we also have that
$ w'_k \in \LF $ and $(w'_{k-1},w'_k)$   is an edge of   $\mathcal{G}(\LF,\CST), k \leq K+2l$.\\
  Moreover we can check that 
  \begin{equation*}
  w'_{K+2l} = \nv + l \bar{c} + \sum_{k \leq K} c_k - l \bar{c} = \mathbf{m} 
  \end{equation*}
   Therefore $\nv$ and $\mathbf{m}$ are connected in $\mathcal{G}(\LF,\CST)$ and the conclusion follows since the choice of $\mathbf{m}$ is arbitrarily in $\LF$ and $\nv$ any point in $\N^A$.  
\item[ii)]  Let $\nv \in \N^A, \mathbf{m} \in \LF$. Since $\CST$ is a lattice basis there exists $K < \infty$ and $\cv_1,...,\cv_K \subseteq (\CST \cup - \CST )^K $ such that if we define recursively :
\begin{equation}
 w_0 = \nv, \quad w_k = \theta_{\cv_k} w_{k-1}
\end{equation}
then we have that $w_K = \mv $
\end{itemize}
Observe that w.l.o.g there exists $K^+$ s.t. $\cv_k \in \CST$ for all $ k \leq K^+ \ $ and 
$\cv_k \in \ -\CST  \ ,k \in \{ K^+ +1,...,A\} $. Applying the hypothesis one can check directly that $\{ w_{k} \}_{0 \leq k \leq K}$ is a path which connects $\nv$ to $\mv$ in $\mathcal{G}(\LF,\CST) $.
\end{proof}

\noindent
{\large \bf References}

\bibliographystyle{plain}
\bibliography{Ref}

\begin{thebibliography}{10}

\bibitem{Bern32}
S.~Bernstein.
\newblock Sur les liaisons entre les grandeurs al\'eatoires.
\newblock In {\em Vehr. des intern. Mathematikerkongress Z\"urich}, volume~I.
  1932.

\bibitem{carlen88}
{E.A.} Carlen and {E.} Pardoux.
\newblock Differential calculus and integration by parts on {P}oisson space.
\newblock In S.~Albeverio, Ph. Blanchard, and D.~Testard, editors, {\em
  Stochastics, Algebra and Analysis in Classical and Quantum Dynamics},
  volume~59 of {\em Mathematics and Its Applications}, pages 63--73. Springer,
  1990.

\bibitem{CMT92}
J.P. {Carmichael}, J.C. {Masse}, and R.~{Theodorescu}.
\newblock {Processus gaussiens stationnaires r\'eciproques sur un intervalle.}
\newblock {\em {C. R. Acad. Sci., Paris, S\'er. I}}, 295:291--293, 1982.

\bibitem{Cha72}
S.C. Chay.
\newblock On quasi-{M}arkov random fields.
\newblock {\em Journal of Multivariate Analysis}, 2(1):14--76, 1972.

\bibitem{Chen}
{L.H.Y.} Chen.
\newblock Poisson approximation for dependent trials.
\newblock {\em The Annals of Probability}, 3(3):534--545, 1975.

\bibitem{Cl91}
{J.M.C.} Clark.
\newblock A local characterization of reciprocal diffusions.
\newblock {\em Applied Stochastic Analysis}, 5:45--59, 1991.

\bibitem{LMR}
G.~Conforti, R.~Murr, C.~L\'eonard, and S.~R{\oe}lly.
\newblock Bridges of {M}arkov counting processes. {R}eciprocal classes and
  duality formulae.
\newblock Preprint, available at
  http://users.math.uni-potsdam.de/$\sim$roelly/, 2014.

\bibitem{Cre93}
{N.A.C.} Cressie.
\newblock {\em {Statistics for Spatial Data}}.
\newblock Wiley, 1993.

\bibitem{CruZa91}
A.B. Cruzeiro and J.C. Zambrini.
\newblock Malliavin calculus and {E}uclidean quantum mechanics. {I}.
  {F}unctional calculus.
\newblock {\em Journal of Functional Analysis}, 96(1):62--95, 1991.

\bibitem{DP91}
P.~{Dai Pra}.
\newblock A stochastic control approach to reciprocal diffusion processes.
\newblock {\em Applied Mathematics and Optimization}, 23(1):313--329, 1991.

\bibitem{Koppe13}
J.A. De~Loera, R.~Hemmecke, and M.~K{\"o}ppe.
\newblock {\em Algebraic and Geometric Ideas in the Theory of Discrete
  Optimization}.
\newblock MOS-SIAM Series on Optimization. SIAM, 2013.

\bibitem{JaShi03}
J.~Jacod and A.N. Shiryaev.
\newblock {\em Limit theorems for stochastic processes}.
\newblock Grundlehren der mathematischen Wissenschaften. Springer, 2003.

\bibitem{Jam70}
B.~Jamison.
\newblock Reciprocal processes: The stationary {G}aussian case.
\newblock {\em The Annals of Mathematical Statistics}, 41(5):1624--1630, 1970.

\bibitem{Jam74}
B.~Jamison.
\newblock Reciprocal processes.
\newblock {\em Zeitschrift für Wahrscheinlichkeitstheorie und Verwandte
  Gebiete}, 30(1):65--86, 1974.

\bibitem{Jam75}
{B.} Jamison.
\newblock The {M}arkov processes of {S}chrödinger.
\newblock {\em Zeitschrift für Wahrscheinlichkeitstheorie und Verwandte
  Gebiete}, 32(4):323--331, 1975.

\bibitem{Kre88}
{A. J.} Krener.
\newblock Reciprocal diffusions and stochastic differential equations of second
  order.
\newblock {\em Stochastics}, 107(4):393--422, 1988.

\bibitem{Kre97}
{A.J.} Krener.
\newblock Reciprocal diffusions in flat space.
\newblock {\em Probability Theory and Related Fields}, 107(2):243--281, 1997.

\bibitem{LRZ}
C.~L{\'e}onard, S.~R{\oe}lly, and J.C. Zambrini.
\newblock Temporal symmetry of some classes of stochastic processes.
\newblock 2013.
\newblock Preprint, available at
  http://users.math.uni-potsdam.de/$\sim$roelly/.

\bibitem{Lev97}
{B.C.} Levy.
\newblock Characterization of multivariate stationary {G}aussian reciprocal
  diffusions.
\newblock {\em Journal of Multivariate Analysis}, 62(1):74 -- 99, 1997.

\bibitem{LeKre93}
{B.C.} Levy and {A.J.} Krener.
\newblock Dynamics and kinematics of reciprocal diffusions.
\newblock {\em Journal of Mathematical Physics}, 34(5), 1993.

\bibitem{Mec67}
J.~Mecke.
\newblock Stationäre zufällige {M}aße auf lokalkompakten {A}belschen
  {G}ruppen.
\newblock {\em Zeitschrift für Wahrscheinlichkeitstheorie und Verwandte
  Gebiete}, 9(1):36--58, 1967.

\bibitem{Murr12}
R.~Murr.
\newblock {\em Reciprocal classes of {M}arkov processes. {A}n approach with
  duality formulae.}
\newblock PhD thesis, Universit{\"a}t Potsdam, 2012.
\newblock Available at opus.kobv.de/ubp/volltexte/2012/6301/pdf/premath26.pdf.

\bibitem{NK}
J.~Neukirch.
\newblock {\em Algebraic Number Theory}.
\newblock Grundlehren der mathematischen Wissenschaften : a series of
  comprehensive studies in mathematics. Springer, 1999.

\bibitem{PrZamb04}
J.C. Privault, N.~Zambrini.
\newblock Markovian bridges and reversible diffusion processes with jumps.
\newblock {\em Annales de l'{I}nstitut {H}enri {P}oincar{\'e} (B),
  {P}robababilit\'es et Statistiques}, 40(5):599--633, 2004.

\bibitem{Roell13}
S.~R{\oe}lly.
\newblock Reciprocal processes. {A} stochastic analysis approach.
\newblock In V.~Korolyuk, N.~Limnios, Y.~Mishura, L.~Sakhno, and G.~Shevchenko,
  editors, {\em Modern {S}tochastics and {A}pplications}, volume~90 of {\em
  Optimization and Its Applications}, pages 53--67. Springer, 2014.

\bibitem{RT02}
S.~R{\oe}lly and M.~Thieullen.
\newblock A characterization of reciprocal processes via an integration by
  parts formula on the path space.
\newblock {\em Probability Theory and Related Fields}, 123(1):97--120, 2002.

\bibitem{RT05}
S.~R{\oe}lly and M.~Thieullen.
\newblock Duality formula for the bridges of a brownian diffusion: Application
  to gradient drifts.
\newblock {\em Stochastic Processes and their Applications},
  115(10):1677--1700, 2005.

\bibitem{Schr}
E.~Schr{\"o}dinger.
\newblock {\"U}ber die {U}mkehrung der naturgesetze.
\newblock {\em Sitzungsberichte Preuss. Akad. Wiss. Berlin. Phys. Math.}, 144,
  1931.

\bibitem{Sli62}
I.~M. Slivnjak.
\newblock Some properties of stationary streams of homogeneous random events.
\newblock {\em Teor. Verojatnost. i Primenen.}, 7:347--352, 1962.
\newblock In Russian.

\bibitem{Ste86}
C.~Stein.
\newblock {\em Approximate Computation of Expectations}.
\newblock IMS Lecture Notes. Institute of Mathematical Statistics, 1986.

\bibitem{Th93}
M.~Thieullen.
\newblock Second order stochastic differential equations and non-{G}aussian
  reciprocal diffusions.
\newblock {\em Probability Theory and Related Fields}, 97(1-2):231--257, 1993.

\bibitem{TZ97}
M.~Thieullen and J.~C. Zambrini.
\newblock Symmetries in the stochastic calculus of variations.
\newblock {\em Probability Theory and Related Fields}, 107(3):401--427, 1997.

\bibitem{Wak89}
A.~Wakolbinger.
\newblock A simplified variational characterization of {S}chrödinger
  processes.
\newblock {\em Journal of Mathematical Physics}, 30(12):2943--2946, 1989.

\bibitem{Zam86}
J.C. Zambrini.
\newblock Variational processes and stochastic versions of mechanics.
\newblock {\em Journal of Mathematical Physics}, 27(9):2307--2330, 1986.

\end{thebibliography}
\end{document}